\documentclass [12pt,a4paper]{article}

\usepackage[latin2]{inputenc}

\usepackage{amsmath}
\usepackage{amsthm}
\usepackage{amssymb}
\usepackage{latexsym}
\usepackage{enumerate}

\usepackage{graphics}
\usepackage{tikz}

\newtheorem {theorem}{Theorem}

\newtheorem {lemma}{Lemma}

\newtheorem {proposition}{Proposition}
\newtheorem {definition}{Definition}

\newtheorem* {theorem*}{Theorem}
\newtheorem* {thm*}{Theorem}
\newtheorem* {lemma*}{Lemma}
\newtheorem* {corollary*}{Corollary}
\newtheorem* {prop*}{Proposition}
\newtheorem* {definition*}{Definition}
\newtheorem* {remark*}{Remark}

\def \A       {\mathcal{A}}
\def \E       {\mathcal{E}}
\def \M       {\mathcal{M}}
\def \N       {\mathbb N}
\def \Nn      {\mathbb M}

\def \R       {\mathbb R}
\def \T       {\mathcal T}
\def \ind     {1\!\!1}
\newcommand{\indd}[1]   {\ensuremath{ \ind \left[ #1\right]}}

\def \rbM     {\text{M}}
\def \rbN     {\text{N}}
\def \cF      {\mathcal{F}}
\def \cG      {\mathcal{G}}
\def \cS      {\mathcal{S}}

\newcommand{\abs}[1]{\left|{#1}\right|}

\newcommand{\prob}[1]    {\ensuremath{\mathbf{P}\left(#1\right)}}
\newcommand{\expect}[1]  {\ensuremath{\mathbf{E}\left(#1\right)}}

\newcommand{\condprob}[2]    {\ensuremath{\mathbf{P}\left(#1\,\big|\,#2\right)}}
\newcommand{\condexpect}[2]  {\ensuremath{\mathbf{E}\left(#1\, \middle| \,#2\right)}}

\def \probp    {\mathbf{P}}

\def \toindis  {\buildrel {d}\over{\longrightarrow}}
\def \ntoinf   {\buildrel {n \rightarrow \infty}\over{\longrightarrow}}

\def \phitetsz   {\varphi : [k]\rightarrow [n]}
\def \phiinj     {\varphi : [k]\hookrightarrow [n]}

\newcommand{\homsur}[2]     {\ensuremath{t_{\leq}\left( #1,#2\right)}}
\newcommand{\homind}[2]     {\ensuremath{t_{=}\left( #1,#2\right)}}
\newcommand{\hominj}[2]     {\ensuremath{t_{\leq}^0\left( #1,#2\right)}}
\newcommand{\hominjind}[2]  {\ensuremath{t_{=}^0\left( #1,#2\right)}}

\newcommand{\graphhom}[2]     {\ensuremath{\ind_{\leq} \left[ #1,#2,\varphi \right]}}
\newcommand{\graphhomm}[2]    {\ensuremath{\ind_{=} \left[ #1,#2,\varphi \right]}}

\def \rndperm                {\pi}
\def \rndinfperm             {\tau}
\def \zero                   {\mathbf{0}}
\newcommand{\randperm}[1]    {\ensuremath{\pi \left( #1 \right)}}
\newcommand{\disuniform}[1]  {\ensuremath{\xi \left( #1 \right)}}
\newcommand{\randarray}[3]   {\ensuremath{\left(X_{#1}^{#2}(i,j)\right)_{i,j=1}^{#3}}}
\newcommand{\rndarray}[2]    {\ensuremath{\mathbf{X}_{#1}^{#2}}}

\author{
{\sc Istv\'an Kolossv\'ary} \qquad and  \qquad {\sc  Bal\'azs R\'ath}
\\[5pt]
Institute of Mathematics \\ Budapest University of Technology (BME)
\\
{\sc
Egry J\'ozsef u. 1}
\\
{\sc H-1111 Budapest, Hungary}
\\
e-mail: {\tt \{istvanko,rathb\}{@}math.bme.hu}
}

\title{Multigraph limits and exchangeability}
\begin{document}

\maketitle

\bigskip

\begin{abstract}
The theory of limits of dense graph sequences was initiated by  Lov\'asz and Szegedy in \cite{Lovasz_Szegedy_2006}.
We give a possible generalization of this theory to multigraphs.
 Our proofs are based on the correspondence between dense graph limits and countable,
exchangeable arrays of random variables observed by Diaconis and Janson in \cite{diaconis_janson}. The main ingredient
in the construction of the limit object is  Aldous' representation theorem for exchangeable arrays, see \cite{aldous_exch}.
\end{abstract}

$ $


\tableofcontents

%%%%%%%%%%%%%%%%%%%%%%%%%%%%%%%%%%%%%%%%%%%%%%%%%%%%%%%%%%%%%%%%%%%%%%
%%%Introduction
%%%%%%%%%%%%%%%%%%%%%%%%%%%%%%%%%%%%%%%%%%%%%%%%%%%%%%%%%%%%%%%%%%%%%%

\section{Introduction}\label{section_intro}

In recent years a limiting theory has been developed for dense graph sequences (in dense graphs the number of edges is comparable with $\abs{V(G)}^2$). Roughly speaking, a sequence $(G_n)_{n=1}^{\infty}$ of simple graphs converges if for any fixed \emph{testgraph} $F$, the density of copies of $F$ found in $G_n$ (called the \emph{homomorphism density}) converges as $n \to \infty$. It was shown  in \cite{Lovasz_Szegedy_2006} that the limit object can be represented by a symmetric measurable function $W: [0,1]^2 \to [0,1]$. Such functions are called \emph{graphons}.

In \cite{Lovasz_Szegedy_2006},   Subsection 6.2 the authors briefly discuss the possible generalization of the theory to multigraphs (graphs with multiple and loop edges), pointing out technical issues which arise because the number of edges possibly connecting two vertices in a multigraph is not bounded, which leads to a lack of the compactness properties used in their proofs. They also show that the notion of graphons is not suitable for defining the limits of multigraphs if the testgraphs are also allowed to be multigraphs.

$ $

In this paper we present a generalization of the theory of dense graph limits to multigraphs.

\begin{itemize}
\item In {\bf Section \ref{section_graphok}} we give a possible way to generalize  the notion of the M\"obius transform, homomorphism densities, graphons, gluing of $k$-labeled graphs,  reflection positivity and convergence of graph sequences to multigraphs.  We state the main result of this paper in Theorem \ref{theorem_main_result}, which is an analogue of Theorem 2.2 of \cite{Lovasz_Szegedy_2006} giving equivalent characterizations of the graph parameters arising as limits of homomorphism densities.
 Proposition \ref{uniqueness} guarantees that our collection of observables determines
  the observed multigraph uniquely. In Proposition \ref{proposition_tightness} we give a useful characterization of
 the precompact subsets of the space of limit objects called \emph{multigraphons}, which are of the form $W: [0,1]^2 \times \N_0 \to [0,1]$  .

\item Our methods
 are different from those used in \cite{Lovasz_Szegedy_2006}: In {\bf Section \ref{section_exchangeability}} we make a connection between multigraph limits and the theory of infinite exchangeable arrays of random variables (based on \cite{diaconis_janson} and \cite{austin_survey}): we generate countable random arrays using multigraphs and multigraphons to show that we can interpret the homomorphism densities as probabilities on a special probability space.
 The \emph{multiplicativity} of graph homomorphism densities corresponds to the \emph{dissociated} property of random arrays, convergence of multigraph sequences corresponds to convergence in distribution of random arrays. In \cite{Lovasz_Szegedy_2009} a parallel theory of \emph{consistent countable random graph models} is described: we give a short dictionary of the corresponding concepts in the different terminologies.

\item In {\bf Section \ref{section_aldous_theorem}} we state and prove Theorem \ref{theorem_aldous}, a representation theorem for exchangeable arrays.
 This theorem is stated but not proved in \cite{aldous_exch}, and proofs of variants of Theorem \ref{theorem_aldous} can be found in \cite{austin_survey} and \cite{kallenberg}, but in our opinion the self-contained and streamlined treatment of the proof helps to understand why
 Szemer\'edi's lemma can be replaced by Aldous' representation theorem in the construction of multigraphons.

\item In {\bf Section  \ref{section_proofs}} we prove Theorem \ref{theorem_main_result} following the cyclic structure of the proof of Theorem 2.2 of \cite{Lovasz_Szegedy_2006}. In many cases, the connection with infinite exchangeable arrays makes the proofs more transparent, e.g. the proof of the reflection positivity of multigraphon homomorphism densities became simpler and Azuma's inequality is no longer needed for the proof of the fact that every multigraphon is the limit object of a convergent graph sequence.
\end{itemize}

The methods of this paper can be applied to give a multigraph generalization of Theorem 3.2 of \cite{Lovasz_Szegedy_2009} relating isolate-indifferent graph parameters to random graphons. Also, Aldous' representation theorem can be useful in the description of the limit object of convergent sequences of weighted graphs. In \cite{austin_survey} representation theorems of higher dimensional random exchangeable arrays are used to describe the limit objects of convergent hypergraph sequences.

\medskip
Representation theorems similar to ours can be found in the literature:

The theory describing the limit objects of weighted graph sequences with \emph{uniformly bounded edgeweights} is presented in \cite{Lovasz_Szegedy_200mikor}. The results therein are highly similar to ours (e.g. the limit objects are of form
$W: [0,1]^2 \times \N \to \R$ and a version of Aldous' representation theorem for exchangeable and dissociated arrays is proved using Szemer\'edi partitions), although some
 definitions are different (e.g. in their definition of the gluing of labeled multigraphs the adjacency matrices are summed whereas in our definition (see \eqref{adjmatrix_of_F1F2}) their maximum is considered, and their definition of homomorphism densities is related to the moment sequence of random variables, whereas ours is more related to the distribution function of the same random variables). The condition on the uniform boundedness of edgeweights in \cite{Lovasz_Szegedy_200mikor} (which is needed for certain compactness arguments) can be relaxed: in \cite{SzakacsLaci} it is shown that the limit of a convergent and uniformly $L^p$-bounded sequence of $\R$-valued graphons can itself be represented by a graphon if we only consider homomorphism densities of simple testgraphs in the definition of the convergence of graphons.
 
 The theory of multigraphons described in this paper
  fits into a more general framework
(worked out in \cite{Lovasz_Szegedy_compact} using Szemer\'edi partitions) in which limits of complete graphs
are studied where the edges are labeled by elements from a fixed compact
topological space $\cS$. In the special case when $\cS$ is the one
point compactification $\N \cup \infty $ of the natural numbers the limit objects
are basically equivalent with the ones studied in the present paper.
 The only difference is that for non-tight sequences the symbol $\infty$ appears
with a nonzero probability in the limit object. In general the limit object
is a measurable function from the unit square to the space of probability
distributions on $\cS $. In  Theorem 2.9 and Corollary 3.5 of \cite{austin_survey}  one finds even more general representation theorems of similar flavor for infinite, exchangeable graphs in which
edges are labeled by elements from a Borel space.

It is apparent from the extensive list of related results above that the theory of multigraphons is already implicitly
present in the literature, but in order to write the paper \cite{elatkotos} about the time evolution of the \emph{edge reconnecting model} we needed a reference in which Theorem \ref{theorem_main_result} (giving equivalent characterizations of the multigraph parameters arising as limits of homomorphism densities) is explicitly stated and proved.

$ $

\noindent
{\bf Acknowledgement.}
Bal\'azs R\'ath thanks L\'aszl\'o Lov\'asz and L\'aszl\'o Szak\'acs for introducing him to the theory of dense graph limits
(and raising the question of possible generalizations to multigraphs)
 and B\'alint T\'oth for turning his attention to the connection with exchangeability.
 Istv\'an Kolossv\'ary thanks Domokos Sz\'asz for introducing him to the theory of dense graph limits.
 We also thank the anonymous referees for their useful comments.

This research was  partially supported by the OTKA (Hungarian
National Research Fund) grants
K 60708 and CNK 77778 and Morgan Stanley Analytics Budapest.

\section{Definitions, statement of Theorem \ref{theorem_main_result}}\label{section_graphok}

 In this section we generalize  the definitions of \cite{Lovasz_Szegedy_2006} to multigraphs and state the generalization of Theorem 2.2 of
 \cite{Lovasz_Szegedy_2006}.

Let $\N:=\{1,2,\dots\}$, $\N_0:=\{0,1,2,\dots \}$ and $[k]=\{1,2,\dots,k\}$.

Denote by $\M$ the set of undirected multigraphs (graphs with multiple and loop edges).
Let $F \in \M$ with $\abs{V(F)}=k$. The adjacency matrix of a labeling of the multigraph $F$ with $[k]$
is denoted by $\left(A(i,j)\right)_{i,j=1}^k$, where $A(i,j) \in \N_0$ is the
number of edges connecting the vertices labeled by $i$ and $j$.
$A(i,j)=A(j,i)$ since the graph is undirected
 and  $A(i,i)$ is  two times the number of loop edges at vertex $i$ (thus $A(i,i)$ is an even number).

 An unlabeled multigraph is the equivalence class of labeled multigraphs where two labeled graphs are equivalent if one can be
 obtained by relabeling the other.
  Thus $\M$ is the set of these equivalence classes of multigraphs, which are also called isomorphism types.

 We denote the set of adjacency matrices of multigraphs on $k$ nodes by $\A_k$, thus
 \[\A_k = \left\{  A\in \N_0^{k \times k}\, :\, A^T=A, \, \forall \, i \in [k] \,\;\;
   2\, |\,   A(i,i)  \right\}. \]
  Let $\rbM,\rbN \subseteq \N$.
  \begin{align}
\label{general_index_set_adjacency}
 \A_{\rbN} &:=\left\{  A\in \N_0^{\rbN \times \rbN}\, :\, \forall \, i,j \in \rbN \; A(i,j)=A(j,i) , \, \forall \, i \in \rbN \,\;\;
   2\, |\,   A(i,i)  \right\}\\
\label{general_index_set_M_N_adjacency}
  \A_{\rbM,\rbN}&:=\left\{  A\in \N_0^{\rbM \times \rbN}\, :\, \forall \, i,j \in \rbM \cap \rbN  \; A(i,j)=A(j,i)
    , \, \forall \, i \in \rbM \cap \rbN  \,\;\;
   2\, |\,   A(i,i)  \right\}
    \end{align}

Let $f$ denote a multigraph parameter, that is $f: \M \to \R$. If $F \in \M$ and $A$ is the adjacency matrix of a labeling of $F$, then let
$f(A):=f(F)$. Conversely, if $f: \bigcup_{k=1}^{\infty} \A_k \to \R$ is constant on isomorphism classes, then $f$ defines a multigraph parameter.

If $A, A' \in \A_k$ then we say that $A \leq A'$ if $\forall\, i,j \in [k]\;  A(i,j) \leq A'(i,j)$.

If $A \in \A_k$ denote by $e(A)$ the number of edges:
\[e(A):=\frac12 \sum_{i,j=1}^k A(i,j) \]
Let $\E_k$ denote the set of adjacency matrices of multigraphs with no multiple edges:
\[ \E_k= \left\{\;  A \in \A_k \, : \,  \forall \, i \neq j \; A(i,j) \in \{0, 1\}, \;\; \forall  i \; A(i,i) \in \{0, 2 \}\;  \right\} \]
We say that the multigraph parameter $f$ is \emph{non-defective} (or briefly $f(\infty)=0$) if
\begin{equation}\label{def_nondefected}
\forall \, k \;\; \forall \, A_1,A_2, \dots \in \A_k, \; \lim_{n \to \infty} e(A_n)= +\infty  \quad \implies \quad
 \lim_{n \to \infty} f(A_n)=0
\end{equation}

\begin{definition}\label{def_mobius_transform}
The \emph{Möbius transform} of a function $f: \A_k \to \R$ is  defined by
  \begin{equation}\label{mobius_trans}
    f^{\dagger}(A)= \sum_{E \in \E_k} (-1)^{e(E)}f(A+E).
  \end{equation}
The \emph{inverse Möbius transform} of $g: \A_k \to \R $ is (formally) defined by
  \begin{equation}\label{mobius_inv}
    g^{-\dagger}(A)= \sum_{n =0}^{\infty}
    \sum_{ A' \in \A_k }
    \ind [A' \geq A \, , \, e(A')=n  ]\cdot g(A')               .
  \end{equation}
\end{definition}

 The infinite
 sum defining $g^{-\dagger}(A)$ converges for some $A \in \A_k$ if and only if it converges for all $A \in \A_k$.

If $f$, $g$  are multigraph parameters (i.e. their value is invariant under relabeling of vertices)
 then $f^{\dagger}$, $g^{-\dagger}$ are also multigraph parameters.

\begin{lemma}\label{mobius_inverzio_tetel}
%$ $
%
%
%\begin{enumerate}[(i)]
%\item
%If $f: \A_k \to \R$ is non-defective then $\left(f^{\dagger}\right)^{-\dagger} \equiv f$.
\begin{equation}\label{mobius_inverzios_tetel_eq}
f: \A_k \to \R, \; f(\infty)=0 \quad \implies \quad  \left(f^{\dagger}\right)^{-\dagger} \equiv f
\end{equation}
%\item If $g: \A_k \to \R$ and the infinite sum in \eqref{mobius_inv} converges , then $\left(g^{-\dagger}\right)^{\dagger} \equiv g$. Ez kell ? Ha nem kell, akkor az elozot inkabb keplettel es eqref legyen a ref helyett
%\end{enumerate}
\end{lemma}
\begin{proof}
%\begin{enumerate}
First note that if $A, A'' \in \A_k$ then
\begin{equation} \label{kiejti_kiveve_ha_egyezik}
\sum_{E \in \E_k} \ind[ A'' \geq A+E]\cdot (-1)^{e(E)} =\ind[A=A'']
\end{equation}
\begin{align*}
\left(f^{\dagger}\right)^{-\dagger}(A)
 &=\lim_{N \to \infty} \sum_{n=0}^N \sum_{A' \in \A_k} \sum_{E \in \E_k} \ind [A' \geq A \, , \, e(A')=n  ] \cdot (-1)^{e(E)} f(A'+E)\\
 &\stackrel{\eqref{def_nondefected}}{=}
 \lim_{N \to \infty} \sum_{n=0}^N \sum_{A' \in \A_k} \sum_{E \in \E_k} \ind [A' \geq A \, , \, e(A'+E)=n  ] \cdot (-1)^{e(E)} f(A'+E)\\
&=\lim_{N \to \infty} \sum_{n=0}^N \sum_{A'' \in \A_k} \sum_{E \in \E_k} \ind [A'' \geq A+E \, , \, e(A'')=n  ] \cdot (-1)^{e(E)} f(A'')
 \stackrel{\eqref{kiejti_kiveve_ha_egyezik}}{=} f(A)
\end{align*}
%\begin{multline*}
%\left(g^{-\dagger}\right)^{\dagger}(A)=\\
%\sum_{E \in \E_k} (-1)^{e(E)} \lim_{N \to \infty} \sum_{n=0}^N \sum_{A'' \in \A_k} \ind[A'' \geq A+E\, , \, e(A'')=n]\cdot g(A'')
%\stackrel{\eqref{kiejti_kiveve_ha_egyezik}}{=}g(A)
% \lim_{N \to \infty} \sum_{n=0}^N \sum_{A'' \in \A_k} \sum_{E \in \E_k} (-1)^{e(E)} \ind[A'' \geq A+E\, , \, e(A'')=n]\cdot g(A'')=\\
%\end{multline*}

%\end{enumerate}
\end{proof}
Note that if $f$ is a constant function then $f^{\dagger}\equiv 0$, thus $f(\infty)=0$ is essential for
$\left(f^{\dagger}\right)^{-\dagger} \equiv f$ to hold.

$ $

  Suppose $F,G \in \M$,  $\abs{V(F)}=k$,  $\abs{V(G)}=n$ and denote by $A \in \A_k$ and $B \in \A_n$ the adjacency matrices of $F$ and $G$.

%   We think of $F$ and $G$ as representatives of their isomorphism classes of unlabeled graphs. However, it will be convenient for us to consider a % temporary labeling of a graph just to give its adjacency matrix a simple indexing. The notation $[n]:=\{1,2,\dots,n\}$ will be used for the vertex set % of a graph with labels $1,2,\dots,n$ on $n$ nodes.

%%%%%%%%%%%%%%%%%%%%%%%%%%%
%%%graph_homomorphism
%%%%%%%%%%%%%%%%%%%%%%%%%%%
Now we generalize the notion of \emph{graph homomorphism} to multigraphs. Let $\varphi:[k] \to [n]$.
 \begin{equation}\label{graphhom_roviden}
\graphhom{A}{B}:= \ind \left[ \, \forall i,j \in [k]: A(i,j)\leq B(\varphi(i),\varphi(j))\right]
\end{equation}
\begin{equation}\label{graphhom_roviden=}
\graphhomm{A}{B}:= \ind \left[ \, \forall i,j \in [k]: A(i,j)= B(\varphi(i),\varphi(j))\right]
\end{equation}
We call $\varphi$ a graph homomorphism of $F$ into $G$ if and only if $\graphhom{A}{B}=1$. This is a natural definition of an
edge-preserving mapping for multigraphs, furthermore if $F$ and $G$ are simple graphs, this new definition coincides with the graph homomorphism definition of
\cite{Lovasz_Szegedy_2006}.

\begin{definition}\label{def_hom_densities}
We define the \emph{homomorphism density} of $F$ into $G$ by
  \begin{equation}\label{hom_sur}
    \homsur{F}{G} :=\homsur{A}{B}:=  \frac{1}{n^k} \sum_{\phitetsz}\graphhom{A}{B} \, .
  \end{equation}
If we restrict the summation to injective maps $\phiinj$  and normalize by the number of such maps we get the \emph{injective homomorphism density}
  \begin{equation}\label{hom_inj}
    \hominj{F}{G} :=\hominj{A}{B}:= \frac{1}{n(n-1)\dots (n-k+1)} \sum_{\phiinj} \graphhom{A}{B} \, .
  \end{equation}
We also define the \emph{induced homomorphism density} of $F$ into $G$ by
  \begin{equation}\label{hom_ind}
    \homind{F}{G}:=\homind{A}{B}:= \frac{1}{n^k} \sum_{\phitetsz}\graphhomm{A}{B} \, ,
  \end{equation}
  \begin{equation}\label{hom_inj_ind}
    \hominjind{F}{G}:=\hominjind{A}{B}:= \frac{1}{n(n-1)\dots (n-k+1)} \sum_{\phiinj} \graphhomm{A}{B} \, .
  \end{equation}
If $\abs{V(F)}>\abs{V(G)}$ in \eqref{hom_inj} and \eqref{hom_inj_ind}, then $\hominj{F}{G}:=\hominjind{F}{G}:=0$.
\end{definition}

Even though we have used a particular labeled member of the isomorphism class $F$ and $G$  in Definition \ref{def_hom_densities}, the quantities are  well-defined for unlabeled graphs $F$ and $G$, since relabeling $F$ and $G$ does not change the value. Thus every fixed multigraph $G$ defines the multigraph parameters \homsur{\,\cdot\,}{G}, \hominj{\,\cdot\,}{G}, \homind{\,\cdot\,}{G}, $t_=^0(\,\cdot\, ,G)$.

For fixed $B$ and $\varphi$ we have $\graphhomm{\cdot}{B}^{-\dagger} \equiv \graphhom{\cdot}{B}$ and
 $\graphhom{\cdot}{B}^{\dagger} \equiv \graphhomm{\cdot}{B}$.

%The functions $\graphhom{\cdot}{B}$ and $\graphhomm{\cdot}{B}$  determine each other:
%\begin{align}
%\label{inversemobius1}
%\graphhom{A}{B}&=\sum\limits_{A'\geq A}\graphhomm{A'}{B}\\
%\label{mobius1}
%\graphhomm{A}{B}&=\sum_{E \in \E_k}(-1)^{e(E) } \cdot \graphhom{A+E}{B}
%\end{align}
%\eqref{inversemobius1} is trivial, \eqref{mobius1} can be proved by expanding the identity
%\begin{multline*}
%\graphhomm{A}{B}=\\ \prod_{i\leq j \leq k} \big( \ind [ A(i,j) \leq B(\varphi(i), \varphi(j))]-
%\ind[ A(i,j)+1 \leq B(\varphi(i),\varphi(j))] \big)
%\end{multline*}

%Hence, \eqref{1} and \eqref{2} imply that $t_{=}=t_{\leq}^{\dagger}$ and $t_{\leq}=t_=^{-\dagger}$. From \eqref{3} and \eqref{4} it follows that %$t_=^0 = \left(t_{\leq}^0\right)^{\dagger}$ and $t_{\leq}^0 = \left(t_=^0\right)^{-\dagger}$.

The homomorphism densities inherit this property, thus we have
\begin{align}
\label{hom_mobius_G}
  &\homsur{\cdot}{G}  \equiv \homind{\cdot}{G}^{-\dagger}
  &\homind{\cdot}{G}  \equiv \homsur{\cdot}{G}^{\dagger} \\
  &\hominj{\cdot}{G}  \equiv \hominjind{\cdot}{G}^{-\dagger}
  &\hominjind{\cdot}{G} \equiv \hominj{\cdot}{G}^{\dagger}
\end{align}

\begin{proposition}\label{uniqueness}
 $G \in \M$ is uniquely determined given
$\left(\homsur{F}{G}\right)_{F \in \M}$ and $\abs{V(G)}$.
\end{proposition}
We prove this proposition in Section \ref{section_exchangeability}.

%%%%%%%%%%%%%%%%%%%%%%%%%%%%%%
%%%graphon_graphparameters
%%%%%%%%%%%%%%%%%%%%%%%%%%%%%%
\begin{definition}
A \emph{multigraphon} is a measurable $W:[0,1] \times [0,1] \times \N_0 \to [0,1]$ function satisfying
\begin{align}
\label{W_symetric_eq}
W(x,y,k) &\equiv W(y,x,k), \\
\label{W_nondefective_eq}
 \sum_{k=0}^{\infty} W(x,y,k) &\equiv 1,\\
 \label{W_diag_even}
 W(x,x, 2k+1)& \equiv 0.
\end{align}
\end{definition}
 For every multigraphon $W$ and multigraph $F$ with adjacency matrix $A \in \A_k$ we define
\begin{align}
\label{def_grafon_hom_sur}
 t_{\leq}(F,W)&:=\int_{[0,1]^k} \prod_{i\leq j \leq k}\, \sum_{l=A(i,j)}^{\infty}
W(x_i,x_j,l)\, \mathrm{d} x_1\, \mathrm{d} x_2\, \dots\, \mathrm{d}x_k\\
\label{def_grafon_ind_hom_sur}
 t_{=}(F,W)&:=\int_{[0,1]^k} \prod_{i\leq j \leq k}
W(x_i,x_j,A(i,j))\, \mathrm{d} x_1\, \mathrm{d} x_2\, \dots\, \mathrm{d}x_k
\end{align}
The functions $t_{\leq}(\cdot,W)$ and $t_{=}(\cdot,W)$ are indeed multigraph parameters: their value is invariant under relabeling.
 It is easy to see that
\begin{equation}\label{W_densities_mobius}
 \homsur{\cdot}{W}  \equiv \homind{\cdot}{W}^{-\dagger} \quad \quad \quad \quad
  \homind{\cdot}{W}  \equiv \homsur{\cdot}{W}^{\dagger}
  \end{equation}

If $G$ is a multigraph on $n$ nodes with adjacency matrix $B \in \A_n$, then let
\begin{equation}\label{graphon_generated_byG}
W_G(x,y,k)= \ind \left[ B(\lceil nx \rceil,\lceil ny\rceil )=k \right]
\end{equation}
be the multigraphon generated by $G$. Although the function $W_G$ depends on the choice of  the labeling of $G$,
 the value of $\homsur{\,\cdot\,}{W_G}$ is invariant under relabeling.
It is easy to see that
\begin{equation}\label{genralt_grafon}
\homsur{\,\cdot\,}{G}\equiv \homsur{\,\cdot\,}{W_G}, \quad \homind{\,\cdot\,}{G} \equiv \homind{\,\cdot\,}{W_G}.
\end{equation}

 Call a multigraph parameter $f$ \emph{normalized}, if $f(\zero_1)=1$, where $\zero_1$ is the graph with a single node and no edges.

  If $f$ satisfies $f(F_1F_2)=f(F_1)f(F_2)$, where $F_1F_2$ denotes the disjoint union of $F_1$ and $F_2$, then $f$ is \emph{multiplicative}.

If $f$ is normalized and multiplicative then $f(\zero_k)=1$ where $\zero_k$ denotes the graph with $k$ nodes and no edges.

   The multigraph parameters \homsur{\,\cdot\,}{G} and \homsur{\,\cdot\,}{W} are normalized, multiplicative and non-defective.

    The multigraph parameters \homind{\,\cdot\,}{G} and \homind{\,\cdot\,}{W} are neither normalized nor multiplicative.

A \emph{k-labeled multigraph} ($k\in \N_0$) is a finite graph with at least $k$ nodes, of which $k$ are labeled by $1,2,\dots,k$.
 For two $k$-labeled graphs $F_1$ and $F_2$, define $F_1F_2$ as the graph that one gets by taking their disjoint union, then identifying nodes with the same label, and the number of edges connecting two labeled nodes in $F_1 F_2$ is the maximum of the number of edges connecting them in $F_1$ and $F_2$. In the special case $k=0$, $F_1F_2$ is simply the disjoint union of the two graphs. Recall the defining equation \eqref{general_index_set_adjacency} of the set of adjacency matrices of multigraphs indexed by a general subset of $\N$.
  If we label the unlabeled vertices of $F_1$ and $F_2$ using disjoint subsets of $\N$ (thus $V(F_1) \cap V(F_2)=[k]$), and  if $A_1 \in \A_{V(F_1)}$ and $A_2 \in \A_{V(F_2)}$ are the adjacency matrices of $F_1$ and $F_2$ then the adjacency matrix of $F_1F_2$ is $A_1 \vee A_2 \in \A_{V(F_1)\cup V(F_2)}$:
\begin{equation}\label{adjmatrix_of_F1F2}
(A_1 \vee A_2) (i,j)=
\begin{cases}
  \max\{A_1(i,j),A_2(i,j)\} & \text{if } i,j\in [k] \\
  A_l(i,j)                  & \text{if } i,j\in V(F_l)\setminus[k] \\
  0                         & \text{otherwise}
\end{cases}
\end{equation}
If $F_1$ and $F_2$ are $k$-labeled simple graphs then this definition of $F_1F_2$ coincides with that of \cite{Lovasz_Szegedy_2006}.

For any multigraph parameter $f$ and integer $k\geq 0$ we define the \emph{connection matrix} $M(k,f)$ as an infinite matrix, whose rows and columns are indexed by (isomorphism classes of) $k$-labeled multigraphs. Its elements are $f(F_iF_j)$, where $F_i$ corresponds to the row which $F_i$ indexes and $F_j$ to the respective column (so $M(0,f)$ is a dyadic matrix if $f$ is multiplicative).

\begin{definition}\label{def_reflectionpos}
A graph parameter $f$ is \emph{reflection positive} if the connection matrices $M(k,f)$ are positive semidefinite for each $k\geq 0$.
\end{definition}

%%%%%%%%%%%%%%%%%%%%%%%%%%%%%%%%%
%%%graphconvergence_mainresult
%%%%%%%%%%%%%%%%%%%%%%%%%%%%%%%%%

\begin{definition}\label{def_graf_konvergencia}
We say that a sequence $(W_n)_{n=1}^{\infty}$ of multigraphons is \emph{convergent} if $f(F)=\lim\limits_{n\rightarrow\infty}\homsur{F}{W_n}$ exists for every multigraph $F$, moreover $f$ is non-defective.

A sequence $(G_n)_{n=1}^{\infty}$ of multigraphs is convergent if $(W_{G_n})_{n=1}^{\infty}$ is convergent.

Let $\T$ denote the set of graph parameters $f$ arising as limits of multigraph sequences:
\begin{equation}\label{def_of_T_eq}
 f \in \T \quad \iff \quad \exists \left(G_n \right)_{n=1}^{\infty} \;
 \forall \, F \in \M \;\;  f(F)=\lim\limits_{n\rightarrow\infty}\homsur{F}{G_n} \; \text{ and }\; f(\infty)=0
 \end{equation}
\end{definition}
If $G_n$ is the multigraph with one vertex and $n$ loop edges then
$\lim\limits_{n\rightarrow\infty}\homsur{F}{G_n}=1$ for every $F \in \M$, but $f(F) \equiv 1$ does not satisfy $f(\infty)=0$, so
the sequence $(G_n)_{n=1}^{\infty}$ is not convergent in this case.

\begin{theorem}\label{theorem_main_result}
For a multigraph parameter $f$ the following are equivalent:
  \begin{enumerate}[(a)]
    \item $f \in \T.$
    \item There exists a multigraphon $W$ for which $f(\,\cdot\,)=t_{\leq}(\cdot,W)$.
    \item  $f$ is normalized, multiplicative, non-defective and reflection positive.
    \item  $f$ is normalized, multiplicative, non-defective and $f^{\dagger}\geq 0$.
  \end{enumerate}
\end{theorem}
We prove this theorem in Section \ref{section_proofs}.

We say that a sequence $(W_n)_{n=1}^{\infty}$ of multigraphons is \emph{tight} if
\begin{align}
\label{W_tight_nondiag}
\lim_{m \to \infty} \max_{n} \int_0^1 \int_0^1 \sum_{k=m}^{\infty} W_n(x,y,k) \, \mathrm{d} x \, \mathrm{d} y &=0 \\
\label{W_tight_diag}
\lim_{m \to \infty} \max_{n} \int_0^1  \sum_{k=m}^{\infty} W_n(x,x,k) \, \mathrm{d} x &=0
\end{align}
\begin{proposition}\label{proposition_tightness}
$ $

\begin{enumerate}[(i)]
\item \label{convergent->tight}
 A convergent sequence of multigraphons is tight.
\item \label{tight->convergent}
 A tight sequence of multigraphons contains a convergent subsequence.
\end{enumerate}
\end{proposition}
We prove this proposition in Subsection \ref{section_tightness}.

%%%%%%%%%%%%%%%%%%%%%%%%%%%%%%%%%%%%%%%%%%%%%%%%%%%%%%%%%%%%%%%%
%%%Section_exchangeability
%%%%%%%%%%%%%%%%%%%%%%%%%%%%%%%%%%%%%%%%%%%%%%%%%%%%%%%%%%%%%%%%
\section{Vertex exchangeable arrays}\label{section_exchangeability}
 We will use $\rndperm$ to denote a uniformly chosen random permutation of $[n]$ and let $\rndperm |_{[k]}$ be the restriction of the function $\rndperm : [n] \rightarrow [n]$ to $[k]$. Thus $\rndperm |_{[k]}$ is a uniformly chosen injective function from $[k]$ to $[n]$.

 Given a multigraph $G$ with adjacency matrix $B \in \A_n$ we define a random array $\randarray{G}{0}{n}$ using the random permutation $\rndperm$ by
\begin{equation}\label{def_finite_rand_adjmatrix}
  X_G^0(i,j) := B\left( \randperm{i}, \randperm{j} \right).
\end{equation}
Thus $\left(X_G^0(i,j)\right)_{i,j=1}^n$ is a random element of $\A_n$ whose distribution only depends on the isomorphism type of $G$.

Now we introduce random infinite labeled multigraphs.
 Recalling \eqref{general_index_set_adjacency} let $\A_{\N}$ denote the set of adjacency matrices $\left(A(i,j)\right)_{i,j=1}^{\infty}$of countable  multigraphs:
 \begin{equation}\label{A_N_def_formula}
 \A_{\N} = \left\{  A\in \N_0^{\N \times \N}\, :\, \forall \, i,j\;  A(i,j)\equiv A(j,i), \, \forall \, i\,\;
   2\, |\,   A(i,i)  \right\}.
 \end{equation}
  We consider the probability space $(\A_{\N}, \mathcal{F}, \probp)$ where $\mathcal{F}$ is the coarsest sigma-algebra with respect to which $A(i,j)$ is measurable for all $i,j$ and $\probp$ is a probability measure on the measurable space $(\A_{\N}, \mathcal{F})$. We are going to denote the infinite random array with distribution $\probp$ by $\rndarray{}{} = \randarray{}{}{\infty}$. We use the standard notation $\mathbf{X}\sim \mathbf{Y}$ if $\mathbf{X}$ and $\mathbf{Y}$ are identically distributed (i.e., their distribution $\probp$ is identical on $(\A_{\N}, \mathcal{F})$).

$ $

Let $\disuniform{1}, \disuniform{2}, \dots$ be i.i.d. uniformly chosen elements of the set $[n]$.
Given a multigraph $G$ with adjacency matrix $B \in \A_n$
 we define an infinite random array $\rndarray{G}{} = \randarray{G}{}{\infty}$ by
\begin{equation}\label{def_infinite_rand_adjmatrix}
  X_G(i,j) := B\left( \disuniform{i}, \disuniform{j} \right).
\end{equation}

The distribution of \rndarray{G}{} is a probability measure $\probp_G$ on the measurable space $\left(\A_{\N}, \mathcal{F}\right)$. Clearly, the distribution of \rndarray{G}{} depends only on the isomorphism class of $G$.

 Now we are in a position to give new probabilistic meaning to the quantities defined in Definition \ref{def_hom_densities}, following \cite{diaconis_janson}. If $F$ is a multigraph with adjacency matrix $A$ indexed by $[k]$, then it is straightforward to check that
\begin{align}
  \hominj{F}{G} &= \prob{\forall i,j \leq k: A(i,j) \leq X_G^0(i,j)}
 \label{hominj_valszamosan} \\
  \hominjind{F}{G} &= \prob{\forall i,j \leq k: A(i,j) = X_G^0(i,j)}
 \label{hominjind_valszamosan} \\
  \homsur{F}{G} &= \prob{\forall i,j \leq k: A(i,j) \leq X_G(i,j)}
   \label{homsur_valszamosan} \\
  \homind{F}{G} &= \prob{\forall i,j \leq k: A(i,j) = X_G(i,j)}
   \label{homind_valszamosan}
\end{align}
For \eqref{hominj_valszamosan} and \eqref{hominjind_valszamosan} we of course need $V(F)\leq V(G)$.

We can also define an infinite random array using a multigraphon $W$.
 \begin{definition}\label{def_X_W}
 Let $U_i$ ($i=1,2,\dots$) be independent random variables uniformly distributed in $[0,1]$. Given $(U_i)_{i=1}^{\infty}$ we define the array $\rndarray{W}{} = \randarray{W}{}{\infty}$ as follows: with probability $W(U_i,U_j,k)$ let $X_W(i,j)=k$.
 \end{definition}

 From this construction and the definition of $W_G$ in \eqref{graphon_generated_byG} it immediately follows that \[\rndarray{W_G}{}\sim \rndarray{G}{}.\]

For every multigraphon $W$ and  multigraph $F$ with adjacency matrix $A$ we have
\begin{align}
  \homsur{F}{W} &= \prob{\forall i,j\leq k: A(i,j)\leq X_W(i,j)} \label{homsur_grafon_valszamosan} \\
  \homind{F}{W} &= \prob{\forall i,j\leq k: A(i,j) = X_W(i,j)}   \label{homind_grafon_valszamosan}
\end{align}

%%%%%%%%%%%%%%%%%%%%%%%%%%%%%%%%%%%%%
%%%Kolmogorov_definitions
%%%%%%%%%%%%%%%%%%%%%%%%%%%%%%%%%%%%%

Here we recall the well-known Kolmogorov extension theorem (\cite{lamperti}, Section 1.4, Theorem 1.):

\begin{lemma}\label{theorem_kolmogorov}
If \randarray{n}{}{n} is the adjacency matrix of a random labeled graph $G_{[n]}$ for all $n$, moreover the consistency condition
\begin{equation}\label{kolmogorov_consistency_condition}
  \randarray{n}{}{m} \sim \randarray{m}{}{m}
\end{equation}
holds for all $m\leq n$ (i.e, $G_{[m]}$ has the same distribution as the subgraph of $G_{[n]}$ spanned by the vertices labeled $1, 2, \dots,m$), then there exists a countable random graph (that is a probability measure on $\left( \A_{\N}, \mathcal{F}\right)$) with adjacency matrix $\rndarray{}{}=\randarray{}{}{\infty}$ such that
\begin{equation}\label{extended_cosistency}
  \randarray{}{}{n} \sim \randarray{n}{}{n}.
\end{equation}
Moreover the distribution of \rndarray{}{} is the unique probability distribution on $\left( \A_{\N}, \mathcal{F}\right)$ for which \eqref{extended_cosistency} holds for all $n$.
\end{lemma}

$ $

\begin{proof}[Proof of Proposition \ref{uniqueness}]

 Assume given  $\abs{V(G)}=n$ and
 $\homsur{F}{G}$ for all $F \in \M$.
  We want to prove that this information uniquely
determines the isomorphism type $G$.
By \eqref{hom_mobius_G} we  may assume given
$\left( \homind{F}{G} \right)_{F \in \M}$.
By \eqref{homind_valszamosan} we know the distribution of
 $\left(X_G(i,j)\right)_{i,j=1}^k$ for all $k$.
By Lemma \ref{theorem_kolmogorov} we may assume given $\mathbf{X}_G$.

 Denote by $B \in \A_n$ the adjacency
matrix of a labeling of $G$. Define an equivalence relation $\simeq$ on $[n]$
by
\[ i \simeq j \quad \iff \quad \forall \, k \in [n]\; \; B(i,k)=B(j,k). \]

 Let $\mathcal{V}$ denote the set of $\simeq$-equivalence classes.

Define $B_{\simeq} \in \A_{\mathcal{V}}$
(see \eqref{general_index_set_adjacency})
by $B_{\simeq}(I,J)=B(i,j)$ where $I,J \in \mathcal{V}$ and $i \in I$, $j \in J$.

For $I \in \mathcal{V}$ let $P_{\simeq}(I):=\frac{\abs{I}}{n}$.
%Thus $\sum_{I \in \mathcal{V}} P_{\simeq}(I) =1$.

The isomorphism type $G$ can be recovered given $B_{\simeq}$,
$P_{\simeq}$
 and $n$.

  Now we show that $B_{\simeq}$ and
$P_{\simeq}$ can be recovered given $\mathbf{X}_G$.

Define a (random) equivalence relation $\cong$ on $\N$ by
\[ i \cong j \quad \iff \quad \forall \, k \in \N \; \; X_G(i,k)=X_G(j,k). \]

Let $\tilde{\mathcal{V}}$ denote the set of $\cong$-equivalence classes.

Define $B_{\cong} \in \A_{\tilde{\mathcal{V}}}$
by $B_{\cong}(\tilde{I},\tilde{J})=B(i,j)$ where $\tilde{I},\tilde{J} \in \tilde{\mathcal{V}}$ and $i \in \tilde{I}$, $j \in \tilde{J}$.
%For $\tilde{I} \in \tilde{\mathcal{V}}$ let

Recalling \eqref{def_infinite_rand_adjmatrix} it is easy to see that
\[ \prob{i \cong j \iff   \disuniform{i} \simeq
  \disuniform{j}}=1 \quad
\text{ and }\quad  \prob{\abs{\mathcal{V}}=\abs{\tilde{\mathcal{V}}}}=1,\] since almost surely every element of $[n]$ will appear as the value of $\disuniform{i}$
for some $i \in \N$. If we define $\tilde{I}:=\{ i \in \N: \disuniform{i} \in I\}$, i.e. we
 label the elements of $\tilde{\mathcal{V}}$ using the corresponding elements of $\mathcal{V}$
 then we have $ B_{\simeq}(I,J)=B_{\cong}(\tilde{I}, \tilde{J})$ for
all $I,J \in \mathcal{V}$  by this definition and
\[ \prob{P_{\simeq}(I)=\lim_{N \to \infty} \frac{1}{N} \sum_{i=1}^N \ind[ i \in \tilde{I}]}=1 \] by the law of large numbers.
\end{proof}

$ $

\begin{definition}
A random array $\rndarray{}{} = \randarray{}{}{\infty}$ is \emph{vertex exchangeable} if
\begin{equation}\label{vertex_exch}
  \left(X(\rndinfperm(i),\rndinfperm(j))\right)_{i,j=1}^{\infty} \sim \randarray{}{}{\infty}
\end{equation}
for all finitely supported permutations $\rndinfperm : \N\rightarrow \N$.
\end{definition}
We make the assumption of finite support only because working with all permutations
introduces the additional technicalities of working with an uncountable group; however, with
the right conventions these are routinely surmountable, and the resulting theory is easily seen
to be equivalent. In fact it directly follows from the above definition and the uniqueness assertion of Kolmogorov's extension theorem (Lemma
\ref{theorem_kolmogorov}) that \eqref{vertex_exch} holds for any explicitly described permutation $\rndinfperm : \N\rightarrow \N$ with an infinite support.

The fact that \rndarray{G}{} is vertex exchangeable easily follows from
 $\left( \disuniform{i} \right)_{i=1}^{\infty} \sim \left( \disuniform{\tau(i)} \right)_{i=1}^{\infty}$.
Similarly,   $\left( U_{i} \right)_{i=1}^{\infty} \sim \left( U_{\tau(i)} \right)_{i=1}^{\infty}$ implies that
 \rndarray{W}{} also satisfies \eqref{vertex_exch}.

By the uniqueness part of Lemma \ref{theorem_kolmogorov} the property \eqref{vertex_exch} is equivalent to
\begin{equation}\label{vertex_exch_equiv}
  \prob{\forall i,j\leq n: X(i,j)=A(i,j)} = \prob{\forall i,j\leq n: X(\rndinfperm(i),\rndinfperm(j))=A(i,j)}
\end{equation}
for all $n\in \N$,  $A\in \A_n$ and  $\rndinfperm$ such that $n'\geq n \implies \rndinfperm(n')=n'$.

Vertex exchangeability has several different names: in \cite{aldous_exch} \rndarray{}{} is called \emph{weakly exchangeable}, in \cite{diaconis_janson} the term \emph{jointly exchangeable} is used, we call \rndarray{}{} vertex exchangeable because the distribution of a countable random graph with adjacency matrix \rndarray{}{} is invariant under any relabeling of the vertices. In Section 2.4 of \cite{Lovasz_Szegedy_2009} the distribution of a vertex exchangeable countable random graph is referred to as \emph{consistent} and \emph{invariant}.

We extend \randarray{G}{0}{n} defined by \eqref{def_finite_rand_adjmatrix} into  $\rndarray{G}{0}=\randarray{G}{0}{\infty}$ by  defining $X_G^0(i,j)=0$ if $i>n$ or $j>n$. Note that the extended \rndarray{G}{0} satisfies \eqref{vertex_exch} for permutations $\rndinfperm : \N \rightarrow \N$ for which $\rndinfperm(n')=n'$ if $n'\geq n$.

%\eqref{vertex_exch_equiv}, \eqref{homind_valszamosan}, \eqref{homind_grafon_valszamosan} and the fact that $t_{=}(\cdot,G)$ and $t_{=}(\cdot,W)$ are multigraph parameters.

\begin{definition}\label{def_dissociated}
Call an infinite array $\rndarray{}{} = \randarray{}{}{\infty}$ \emph{dissociated} if for all $n$: $\randarray{}{}{n}$ is independent of $\left( X(i,j)\right)_{i,j=n+1}^{m}$ for each $m>n$.
\end{definition}

It is easy to see that \rndarray{G}{} and \rndarray{W}{}  are dissociated. We have taken the terminology dissociated over from \cite{aldous_exch}, in \cite{Lovasz_Szegedy_2009} it is referred to as the \emph{local} property of the distribution of the random graph: the distribution of subgraphs spanned by disjoint vertex sets are independent.

\subsection{Convergence of random arrays}

We say that a sequence of infinite  arrays \randarray{n}{}{\infty} \emph{converges in distribution} if
\begin{align}
\label{eq_eobankonv} & \forall\, k\; \forall \, A \in \A_k \;  \lim\limits_{n\to \infty} \prob{\forall i,j\leq k: A(i,j)=X_n(i,j)}=g(A),\\
\label{nemdefektiv} &\forall \, k \; \sum_{A \in \A_k} g(A)=1
\end{align}
for some $g: \bigcup_{k=1}^{\infty} \A_k \to \R_{+}$.
%exists for all $k$ and $A\in \A_k$, moreover $g^{-\dagger}(\zero_k)=1$ for all $k$.

Alternatively we might say that \randarray{n}{}{\infty} converges in distribution if and only if
\randarray{n}{}{k} converges in distribution to some random element
 of $\A_k$ for all $k$.

 By Lemma \ref{theorem_kolmogorov} there exists a random infinite array $\rndarray{}{}=\randarray{}{}{\infty}$ such that for all $k$ and $A\in\A_k$
\begin{equation}\label{eobankonv_detailed}
  \prob{\forall i,j\leq k: A(i,j)=X_n(i,j)} \ntoinf \prob{\forall i,j\leq k: A(i,j)=X(i,j)}
\end{equation}
In this case we say that $\rndarray{n}{} \,\toindis \rndarray{}{}$. If $\rndarray{n}{}$ is exchangeable for all $n$, then $\rndarray{}{}$ is also exchangeable. If $\rndarray{n}{}$ is dissociated for all $n$, then $\rndarray{}{}$ is also dissociated.

\begin{lemma}\label{lemma_graphkonv_konvindist}
For any multigraph sequence $(G_n)_{n=1}^{\infty}$ and any multigraph parameter $f$ it is equivalent that
\begin{enumerate}[(a)]
  \item \label{a11} $(G_n)_{n=1}^{\infty}$ converges according to Definition \ref{def_graf_konvergencia} and
\begin{equation}\label{kispal}
\forall\, F \in \M \;\;   \lim\limits_{n\to\infty} \homsur{F}{G_n} = f(F).
 \end{equation}
  \item \label{b11} The sequence of infinite random arrays $(\rndarray{G_n}{})_{n=1}^{\infty}$ converges in distribution and
      \begin{equation}\label{f_leq_valszamosan}
      \forall \, k \; \forall \, A \in \A_k \; \; \lim_{n\to\infty} \prob{\forall i,j\leq k: A(i,j) \leq X_{G_n}(i,j)} = f(A).
      \end{equation}
\end{enumerate}
\end{lemma}
\begin{proof}
$ $

$\eqref{a11} \implies \eqref{b11}$: If $\lim_{n \to \infty} t_{\leq}(\cdot,G_n)=f(\cdot)$ then by \eqref{mobius_trans} and \eqref{W_densities_mobius} we get
\begin{equation}\label{csacsacsa}
\lim\limits_{n\rightarrow \infty}\homind{\cdot}{G_n} = f^{\dagger}(\cdot).
\end{equation}
 By
 $f(\infty)=0$ (which is assumed in Definition \ref{def_graf_konvergencia}) we obtain
 \[\sum_{A \in \A_k} f^{\dagger}(A) =  \left(f^{\dagger}\right)^{-\dagger}(\zero_k) \stackrel{\eqref{mobius_inverzios_tetel_eq}}{=} f(\zero_k)= \lim_{n \to \infty} t_{\leq}(\zero_k,G_n)= 1.\]
 Using \eqref{homind_valszamosan} and \eqref{csacsacsa} we get that $\mathbf{X}_n:=\mathbf{X}_{G_n}$ and $g:=f^{\dagger}$ satisfy \eqref{eq_eobankonv} and \eqref{nemdefektiv},  thus $(\rndarray{G_n}{})_{n=1}^{\infty}$ converges in distribution.
\eqref{f_leq_valszamosan} follows from \eqref{homsur_valszamosan} and \eqref{kispal}.

$ $

$\eqref{b11} \implies \eqref{a11}$: $\lim\limits_{n\to\infty} \homsur{F}{G_n} = f(F)$ follows from
\eqref{homsur_valszamosan} and \eqref{f_leq_valszamosan}.

Let $\rndarray{G_n}{} \,\toindis \rndarray{}{}$ where
 $\prob{\forall i,j\leq k: A(i,j) \leq X(i,j)} = f(A)$.
 In order to prove $f(\infty)=0$ let $k \in \N$ and $A_1, A_2, \dots \in \A_k$ such that $e(A_n) \to +\infty$.
Now $\left( X(i,j)\right)_{i,j=1}^k$ is a random element of $\A_k$ thus
the random variable $e \left( \left(X(i,j)\right)_{i,j=1}^k \right)$ is almost surely finite. It is easy to see that
\[ \left\{ \forall i,j\leq k: A_n(i,j) \leq X(i,j) \right\} \subseteq
 \left\{ e(A_n) \leq e \left( \left(X(i,j)\right)_{i,j=1}^k \right) \right\}.
 \]
 Since the probability of the r.h.s. goes to $0$ as $n \to \infty$, we obtain that $f$ is non-defective.
\end{proof}

$ $

Now we prove
\begin{equation}
 \abs{V(G_n)}\to\infty \quad \implies \quad
 \left( \rndarray{G_n}{0} \toindis \rndarray{}{}\; \iff \; \rndarray{G_n}{} \toindis \rndarray{}{} \right)
 \end{equation}
 By  \eqref{hominjind_valszamosan}, \eqref{homind_valszamosan} and \eqref{eq_eobankonv}  we only need to show that
 \begin{equation} \label{homdensities_converge_together}
\abs{V(G_n)}\to\infty \quad \implies \quad \forall \, F \in \M\;\; \lim_{n \to \infty} \abs{\homind{F}{G_n} - \hominjind{F}{G_n}} = 0.
 \end{equation}
  We show that if $V(F)=k$ and $V(G)=n$ then
 \[ \abs{\homind{F}{G} - \hominjind{F}{G}}\leq \frac{1}{n} \binom{k}{2} . \]
We might assume $k<n$.  Recalling the formulae \eqref{def_finite_rand_adjmatrix} and \eqref{def_infinite_rand_adjmatrix} one can see that
 $\left( X_{G}^0(i,j)\right)_{i,j=1}^{k}$ has the same distribution as $\left( X_{G}(i,j)\right)_{i,j=1}^{k}$ under the condition
 \[\abs{ \left\{ \disuniform{1}, \dots, \disuniform{k} \right\} }=k.\]
For any two events $A$ and $B$ in any probability space
\begin{equation}\label{egyenlotlensegunk}
  \abs{\prob{A}-\condprob{A}{B}}\leq 1-\prob{B}.
\end{equation}
Using this inequality we get the desired
\begin{eqnarray*}
  &&\abs{\homind{F}{G} - \hominjind{F}{G}}\leq \prob{ \abs{ \left\{ \disuniform{1}, \dots, \disuniform{k} \right\} }<k  } = \\
  && \prob{\bigcup_{i<j\leq k} \disuniform{i}=\disuniform{j}} \leq \sum\limits_{i<j\leq k} \prob{\disuniform{i}=\disuniform{j}} = \frac{1}{n} \binom{k}{2} .
\end{eqnarray*}

\section{A representation theorem for vertex exchangeable arrays}\label{section_aldous_theorem}
In this chapter we consider vertex exchangeable random elements of $\A_{\N}$.
The proofs work without any change for real-valued, symmetric, vertex exchangeable arrays, which correspond to adjacency
matrices of undirected, infinite, vertex exchangeable \emph{weighted} graphs.

\begin{theorem}\label{theorem_aldous}
Let $\alpha, \left(U_i \right)_{1 \leq i},  \left(\beta_{i,j} \right)_{1 \leq i \leq j}$ be i.i.d. random variables distributed uniformly in $[0,1]$.

\begin{enumerate}[(i)]
\item \label{theorem_aldous_negyvaltozos}
 Given a  vertex exchangeable array \rndarray{}{} (that is: a vertex exchangeable random element of $\A_{\N}$), there exists a measurable function
  $f: [0,
1]^4 \to \N_0$ such that
\[f(a,u_1,u_2,b) \equiv f(a,u_2,u_1,b)\]
and
if we define the random array $\tilde{\mathbf{X}}$ by
\begin{equation}\label{representation_formula_f}
\tilde{X}(i,j):= f(\alpha, U_i, U_j, \beta_{\min(i,j) ,\max(i,j)})
\end{equation}
then we have $\tilde{\mathbf{X}} \sim \mathbf{X}$.

\item \label{theorem_aldous_haromvaltozos}
Moreover, if \rndarray{}{} is also dissociated then there exists a measurable function $g:[0,1]^3 \to \N_0$
  such that $g(u_1,u_2,b)\equiv g(u_2,u_1,b)$ and
    \begin{equation}\label{representof_exch_dissoc}
      \randarray{}{}{\infty} \sim \left( g(U_i, U_j, \beta_{\min(i,j) ,\max(i,j)} )\right)_{i,j=1}^{\infty}.
     \end{equation}
\end{enumerate}
\end{theorem}
Our aim is to give an accessible and self-contained proof.
 Proofs of different versions of this theorem can be found in the literature:

 In Theorem $1.4$ of \cite{aldous_exch} a proof of the analogue of \eqref{theorem_aldous_negyvaltozos} is given for row and column exchangeable (RCE) arrays (different permutations can be applied to the rows and columns).
  The variant of \eqref{theorem_aldous_negyvaltozos} we prove (where $\mathbf{X}$ is not RCE, only vertex exchangeable)  is only stated in Theorem 5.1 of \cite{aldous_exch}. Our proof of
 \eqref{theorem_aldous_negyvaltozos} follows the structure of the proof of Theorem $1.4$ of \cite{aldous_exch} but our proofs of Lemmas \ref{rows_cond_indep_lemma}, \ref{edges_indep_lemma} and \ref{cond_indep_fapipa} use methods from \cite{austin_survey}.

   \eqref{theorem_aldous_haromvaltozos} is only stated and proved for RCE arrays in \cite{aldous_exch}, but the proof works in the vertex exchangeable case as well.
   The proof of a more general version of Theorem \ref{theorem_aldous} is Chapter 7 of \cite{kallenberg}.
    Variants of  \eqref{theorem_aldous_negyvaltozos} and \eqref{theorem_aldous_haromvaltozos} for random infinite exchangeable arrays corresponding to simple graphs are proved in Theorems 3.1 and 3.2 in \cite{Lovasz_Szegedy_2009}.

%%%%%%%%%%%%%%%%%%%%%%%%%%%%%%%%%%%%%%%%%%%%%%%%%%%%%%%%%%%%%%%%
%%%Subsection_stochastic preliminaires
%%%%%%%%%%%%%%%%%%%%%%%%%%%%%%%%%%%%%%%%%%%%%%%%%%%%%%%%%%%%%%%%
\subsection{Preliminaries}
In this subsection we state some less known facts of probability theory needed for the proof of Theorem \ref{theorem_aldous}.

Assume given a probability space $\left( \Omega, \mathcal{F}, \probp \right)$ and sub-$\sigma$-algebra $\mathcal{G} \subseteq \cF$.
We denote the set of $\mathcal{G}$-measurable functions by $\mathbf{m}\mathcal{G}$.
Let $Y \in L^1\left( \Omega, \mathcal{F}, \probp \right) $ denote a real-valued random variable on $\Omega$ with finite expectation.
The definition and  basic properties of $\condexpect{Y}{\mathcal{G}}$, the conditional expectation of
$Y $  with respect to $\mathcal{G}$ are
 given in Chapter 9 of \cite{williams}. The defining property of $\condexpect{Y}{\mathcal{G}}$ is
 \begin{equation}\label{def_prop_of_cond_expectation}
 Y' \in \mathbf{m} \cG \; \text{ and } \; \forall Z \in L^{\infty}(\Omega,\cG) \;\, \expect{Y' Z}=\expect{Y Z}
 \;\; \iff \;\; \prob{Y'=\condexpect{Y}{\mathcal{G}}}=1
 \end{equation}

  We are going to use \emph{Steiner's theorem for conditional expectations}:
 \begin{equation}\label{steiner}
 Z \in \mathbf{m}\mathcal{G} \;\; \implies \;\;
 \expect{ \left(Y-Z\right)^2}=
 \expect{ \left( Y-\condexpect{ Y}{\mathcal{G}}\right)^2}  +
 \expect{ \left( \condexpect{Y}{\mathcal{G}}-Z\right)^2}
 \end{equation}

Let $\left( \mathcal{F}_v \right)_{v \in V}$ be a finite or countably infinite family of sub-$\sigma$-algebras of $\mathcal{F}$.
 We say that $\left( \mathcal{F}_v \right)_{v \in V}$ are \emph{conditionally independent} given $\mathcal{G}$ if for all $n \in \N$,
 $v_1,\dots, v_n \in V$, all $i \in [n]$, $Y_i \in L^{\infty}(\Omega,\mathcal{F}_{v_i} )$ we have
 \begin{equation}\label{defining_eq_of_cond_independence}
  \condexpect{ \prod_{i=1}^n  Y_{i} }{\mathcal{G}}=\prod_{i=1}^n \condexpect{   Y_{i} }{\mathcal{G}}
  \end{equation}
 \begin{lemma}\label{lemma_vegesites}
 If $\cF^1_v \subseteq \cF^2_v \subseteq \dots$ and $\cF_v=\sigma \left( \bigcup_{n=1}^{\infty} \cF^n_v \right)$ for all $v \in V$, then
 $\left( \mathcal{F}_v \right)_{v \in V}$ are conditionally independent given $\mathcal{G}$ if and only if
 $\left( \mathcal{F}^n_v \right)_{v \in V}$ are conditionally independent given $\mathcal{G}$ for all $n \in \N$.
 \end{lemma}
The proof is a standard exercise in measure theory.

Two measurable spaces are \emph{Borel-isomorphic} if there exists a bijection
$\Phi$ between them such that $\Phi$ and $\Phi^{-1}$ are measurable.
A \emph{ Borel space} $\left(S, \mathcal{B}\right)$
 is a measurable space which is Borel-isomorphic to some Borel subset of the real line.

If $Y: \Omega \to S$ is a random variable taking values in the Borel space $\left(S, \mathcal{B}\right)$ then
 the $\sigma$-algebra generated by $Y$ is
$ \sigma(Y):= \left\{ Y^{-1}(B)\, :\, B \in \mathcal{B} \right\}$.
 By 3.13 of \cite{williams}:
 \begin{equation}\label{measurability_trick}
  Z \in \mathbf{m} \sigma(Y) \quad \iff \quad  \exists \; f \in \mathbf{m} \mathcal{B}\,:\, Z=f(Y)
  \end{equation}

Let $\left( Y_v \right)_{v \in V}$ be a finite or countably infinite family of random variables taking values in the Borel spaces
$Y_v \in S_v$.
 We say that $\left( Y_v \right)_{v \in V}$ are conditionally independent given $\cG$
  if $\left( \sigma(Y_v) \right)_{v \in V}$
  are conditionally independent given $\cG$. If $\cG=\sigma(Z)$ where  $Z$ is a random variable taking
   its values in the Borel space $S_Z$ then we say that $\left( Y_v \right)_{v \in V}$ are conditionally independent given
    $Z$.

\begin{lemma}\label{eztkellellenoriznifeltfuggetlenseghez}
Let $X$, $Y$ be random variables taking values in the Borel spaces
$S_X$, $S_Y$. Let $\cG^1 \subseteq \cG^2 \subseteq \dots$ and $\cG=\sigma\left( \bigcup_{n=1}^{\infty} \cG^n \right)$.
 $X$ and $Y$ are conditionally independent given $\cG$
 if and only if for all $n \in \N$, $Z \in L^{\infty}(\Omega, \cG^n)$, $f \in L^{\infty}(S_X)$, $g \in L^{\infty}(S_Y)$ we have
\begin{equation}
\expect{f(X)g(Y)Z}=\expect{\condexpect{f(X)}{\cG}g(Y)Z}
\end{equation}
\end{lemma}
The proof easily follows from the definitions of generated $\sigma$-algebra, conditional expectation and conditional independence.

 If $V=\{v_1,v_2,\dots\}$ is a
finite or countably infinite set then $\left(\N_0^V,\mathcal{B}\right)$ is a Borel space where $\mathcal{B}$ is the $\sigma$-algebra
generated by the finite cylinder sets of $\N_0^V$. A sequence
$\left(Y_v\right)_{v \in V}$ of $\N_0$-valued random variables may be regarded as a single r.v. $\mathbf{Y}$ taking values in the
Borel space $\N_0^V$. In this case
\begin{equation}\label{vegesites_eq}
 \sigma( \mathbf{Y} )= \sigma \left( \bigcup_{n=1}^{\infty} \sigma(Y_{v_1},\dots,Y_{v_n}) \right)
 \end{equation}

\begin{lemma}\label{cond_indep_drop_lemma}
Let $X$, $Y$ and $Z$ be random variables taking values in the Borel spaces
$S_X$, $S_Y$, $S_Z$.
$X$ and $Y$ are conditionally independent given $Z$ if and only if
\begin{equation}\label{cond_indep_drop_formula}
 \condexpect{f(X)}{Y,Z}=\condexpect{f(X)}{Z}
\end{equation}
for each $f \in L^{\infty}(S_X)$.
\end{lemma}
This is Theorem 1.45 in \cite{dellacherie}.

Call a family of random variables $\left(Y_v\right)_{v \in V}$ (taking values in the Borel space $S$)
 \emph{conditionally identically distributed}
given $\cG$ if $\condexpect{ f(Y_v)}{\cG}=\condexpect{ f(Y_w)}{\cG}$ for each $v,w \in V$ and $f \in L^{\infty}(S)$. When
$\cG=\sigma(Z)$ this is equivalent to $\left(Y_v,Z\right)\sim \left(Y_w,Z\right)$ for all $v,w \in V$.

 \begin{lemma}[Identification Lemma] \label{identificationlemma}
Let $\mathbf{Y}=\left(Y_v\right)_{v \in V}$,
$\mathbf{Y^*}=\left(Y^*_v\right)_{v \in V}$ and $Z$ be such that
\begin{enumerate}
\item $\{ Y_v: v \in V \}$ are conditionally independent given
$Z$,
\item $\{ Y^*_v: v \in V \}$ are conditionally independent given
$Z$,
\item for each $v \in V$, $Y_v$ and $Y^*_v$ are conditionally
indentically distributed given $Z$.
\end{enumerate}
Then $(Z,\mathbf{Y}) \sim (Z,\mathbf{Y^*})$
\end{lemma}
The proof of this statement is an easy exercise.

Let $\alpha$ denote a random variable uniformly distributed on
$\lbrack 0, 1 \rbrack$, or $ \alpha \sim U\lbrack 0, 1 \rbrack$. Constructing r.v.'s with prescribed distributions
 using $\alpha$ is called \emph{coding}.

\begin{lemma}[Coding Lemma]\label{codinglemma}

$ $

\begin{enumerate}[ (i) ]
\item \label{codinglemma1}
 Let $Y$ be a random variable taking values in the Borel space
$S_Y$. Then
there exists a measurable function $f: \lbrack 0,1 \rbrack \to
S_Y$ such that $Y \sim f(\alpha)$
\item \label{codinglemma2}
 Suppose further that $Z$ is a random variable
taking values in $S_Z$, and suppose that $\alpha$ is independent
of $Z$. Then there exists a measurable function \[g: S_Z
\times \lbrack 0, 1 \rbrack \to S_Y\] such that $(Z,Y) \sim
(Z, g(Z,\alpha))$.
\item \label{codinglemma3}
 Suppose further that $X_1$ and $X_2$ are random variables taking values in $S_X$,
\[(Z,X_1,X_2,Y) \sim (Z,X_2,X_1,Y)\] and  $\alpha$ is independent
of $(Z,X_1,X_2)$. Then there exists a measurable function
\[h: S_Z \times S_X \times S_X
\times \lbrack 0, 1 \rbrack \to S_Y\]
such that $h(z,x_1,x_2,a) \equiv h(z,x_2,x_1,a)$ and
\[ (Z,X_1,X_2,Y)\sim (Z,X_1,X_2,h(Z,X_1,X_2,\alpha))\]
\end{enumerate}
\end{lemma}
\begin{proof}

$ $

\begin{enumerate}[(i)]
\item First suppose $S_Y=\R$. Let $F(y)=\prob{Y\leq y}$ denote the right-continuous distribution function of $Y$. Let
\[ f(a):=F^{-1}(a):=\min \{y:F(y)\geq a \} \]
Then $Y\sim f(\alpha)$. The extension to the Borel space case is an immediate consequence of the definition of Borel space.
\item Again, first suppose $S_Y=\R$. Let $F(y|z)$ be the conditional distribution function of $Y$ given $Z$. For each $z$  let
$g(z,a):=F^{-1}(a|z)$ be the inverse function. With this definition $g$ has the property $(Z,Y) \sim
(Z, g(Z,\alpha))$. Again, the extension to the case when $S$ is a Borel space is straightforward.
\item Suppose $S_Y=\R$. Let $F(y|z,x_1,x_2)$ be the conditional distribution function of $Y$ given $(Z,X_1,X_2)$. From
$(Z,X_1,X_2,Y) \sim (Z,X_2,X_1,Y)$ it follows that $F(y|z,x_1,x_2)\equiv F(y|z,x_2,x_1)$.

$h(z,x_1,x_2,a):=F^{-1}(a|z,x_1,x_2)$ has the desired properties. The extension to the case when $S$ is a Borel space is straightforward.

\end{enumerate}
\end{proof}

\begin{lemma}
Let $Y \in L^2(\Omega,\cF,\probp)$ and $\cF_1 \subseteq \cF_2 \subseteq \dots$,
$\cF=\sigma\left( \bigcup_{n=1}^{\infty} \cF_n \right)$, $\cG_1 \supseteq \cG_2 \supseteq \dots$,
$\cG= \bigcap_{n=1}^{\infty} \cG_n $. Then
\begin{align}
\label{upward} \condexpect{Y}{\cF_n} &\to Y &\text{almost surely and in $L^2$}\\
\label{downward} \condexpect{Y}{\cG_n} &\to \condexpect{Y}{\cG} &\text{almost surely and in $L^2$}
\end{align}
\end{lemma}
The $L^2$ convergence results easily follow from standard Hilbert space arguments and \eqref{steiner}.
The a.s. convergence results follow from
 L\'evy's 'Upward' and 'Downward' theorems (Theorems 14.2 and 14.4 in \cite{williams}).

\subsection{Proof of Theorem \ref{theorem_aldous}}

For the proof, let $\Nn:= \{ -1,-2,\dots \}$. To simplify our notation we denote $X_{i,j}:=X(i,j)$.
It is easy to see (e.g. by Lemma \ref{theorem_kolmogorov}) that the random array
 $\left(X_{i,j} \right)_{i,j \in \N}$ has an exchangeable extension
$\left(X_{i,j} \right)_{i,j \in \Nn \cup \N}$.

In this section, $\rbM$ and $\rbN$ will always denote \emph{finite} sequences of integers.

 If $\rbM=(i_1,i_2,\dots,i_m)$ and $\rbN=(j_1,j_2,\dots,j_n)$
are sequences of integers and $\left(X_{i,j} \right)_{i,j \in \Nn \cup \N}$ is a random array then $X_{\rbN,\rbM}$ denotes the random array
\[ X_{\rbN,\rbM}:= \left(
\begin{array}{cccc}
X_{i_1,j_1} & X_{i_1,j_2} & \dots  & X_{i_1,j_n} \\
X_{i_2,j_1} & X_{i_2,j_2} & \dots  & X_{i_2,j_n} \\
\vdots      & \vdots      & \ddots & \vdots      \\
X_{i_m,j_1} & X_{i_m,j_2} & \dots & X_{i_m,j_n}  \\
\end{array} \right)
\]
We define the concatenation of the sequences $\rbM$ and
$\rbN$ by \[\rbM \rbN=(i_1,\dots,i_m,j_1,\dots,j_n).\]

\begin{lemma}\label{rows_cond_indep_lemma}
For any vertex exchangeable array $\left(X_{i,j}\right)_{i,j \in \N \cup \Nn}$ the $\N_0^{\Nn}$-valued random variables
 $\left(X_{n,\Nn}\right)_{n \in \N}$  are conditionally
independent given $X_{\Nn,\Nn}$.
\end{lemma}

\begin{proof}
%By \eqref{defining_eq_of_cond_independence} we need to show that $\forall n$ and $f_1,\dots,
%f_n \in L^{\infty}(\N_0^\Nn)$ we have
%\begin{equation}
% \condexpect{\prod_{i=1}^n
%f_i(X_{i,\Nn})}{X_{\Nn,\Nn}}= \prod_{i=1}^n\condexpect{
%f_i(X_{i,\Nn})}{X_{\Nn,\Nn}}
%\end{equation}

By Lemma \ref{lemma_vegesites} and \eqref{vegesites_eq} we only need to
show that for all $\rbM \subseteq \Nn$ and $n \in \N$ the $\N_0^{\rbM}$-valued random variables
 $\left(X_{i,\rbM}\right)_{i \in [n]}$ are conditionally independent given $X_{\Nn,\Nn}$.

%show this for $f_1,\dots,f_n$ with a finite support. We might
%assume that $ f_1,\dots,f_n \in L^{\infty}(\N_0^{\rbM})$ .

 By induction we only
need to prove that for all $\rbM \subseteq \Nn$, $ i \in \N$,  $\tilde{\rbN}
\subseteq \N$ such that $i \notin \tilde{\rbN}$ the random variables $X_{i,\rbM}$ and $X_{\tilde{\rbN},\rbM}$ are
conditionally independent given $X_{\Nn,\Nn}$.

% and for all $f \in L^{\infty}(\N_0^{\rbM})$, $g \in L^{\infty}(\A_{\tilde{\rbN}, \rbM})$ (see \eqref{general_index_set_M_N_adjacency} for the definition of $\A_{\tilde{\rbN}, \rbM}$) we have
%\begin{equation}
%\condexpect{
%f(X_{i,\rbM})g(X_{\tilde{\rbN},\rbM})}{X_{\Nn,\Nn}}=
%\condexpect{ f(X_{i,\rbM})}{X_{\Nn,\Nn}}
%\condexpect{g(X_{\tilde{\rbN},\rbM})}{X_{\Nn,\Nn}}
%\end{equation}
By Lemma \ref{eztkellellenoriznifeltfuggetlenseghez} and \eqref{measurability_trick}
we only need to show that for all finite $\rbM'
\subseteq \Nn$ and all $f \in L^{\infty}(\N_0^{\rbM})$, $g \in L^{\infty}(\A_{\tilde{\rbN}, \rbM})$ and
 $h \in
L^{\infty}(\A_{\rbM', \rbM'})$ (see \eqref{general_index_set_M_N_adjacency})
 we have
\begin{equation}
\expect{f(X_{i,\rbM})g(X_{\tilde{\rbN},\rbM})h(X_{\rbM',\rbM'})}=
\expect{\condexpect{f(X_{i,\rbM})}{X_{\Nn,\Nn}}
g(X_{\tilde{\rbN},\rbM})h(X_{\rbM',\rbM'})}
\end{equation}
We might assume without loss of generality that $\rbM=\rbM \cup
\rbM'=\rbM'$. Thus we have to show
\begin{equation}\label{condexp_1}
\expect{f(X_{i,\rbM})g(X_{\tilde{\rbN},\rbM})h(X_{\rbM,\rbM})}=
\expect{\condexpect{f(X_{i,\rbM})}{X_{\Nn,\Nn}}
g(X_{\tilde{\rbN},\rbM})h(X_{\rbM,\rbM})}
\end{equation}
Now let $\tilde{\rbM} \subseteq \Nn$ be such that
$\abs{\tilde{\rbM}}=\abs{\tilde{\rbN}}$ and $\tilde{\rbM} \cap
\rbM=\emptyset$. We apply a permutation $\tau$ that swaps the
elements of $\tilde{\rbN}$ with those of $\tilde{\rbM}$. Using
exchangeability and the fact that $g(X_{\tilde{\rbM},\rbM})h(X_{\rbM,\rbM})$  is $\sigma \left(X_{\Nn,\Nn}\right)$-measurable we get
\begin{multline}\label{condexp_2}
\expect{f(X_{i,\rbM})g(X_{\tilde{\rbN},\rbM})h(X_{\rbM,\rbM})}=\expect{f(X_{i,\rbM})g(X_{\tilde{\rbM},\rbM})h(X_{\rbM,\rbM})}=\\
\expect{\condexpect{f(X_{i,\rbM})}{X_{\Nn,\Nn}}g(X_{\tilde{\rbM},\rbM})h(X_{\rbM,\rbM})}
\end{multline}
Before applying the permutation $\tau$ again we show that the random variable
 $\condexpect{f(X_{i,\rbM})}{X_{\Nn,\Nn}}$ is invariant under $\tau$. Let $\tilde{\Nn}:=\Nn \setminus \tilde{\rbM}$. Now we prove
\begin{equation}\label{tail_prop}
\condexpect{f(X_{i,\rbM})}{X_{\Nn,\Nn}}=\condexpect{f(X_{i,\rbM})}{X_{\tilde{\Nn},\tilde{\Nn}}}
\end{equation}
$\sigma(X_{\tilde{\Nn},\tilde{\Nn}})  \subseteq \sigma(X_{\Nn,\Nn})$, thus by the repeated application of \eqref{steiner}
we only need to show
\begin{equation}
\label{equal_norm}
\expect{\condexpect{f(X_{i,\rbM})}{X_{\Nn,\Nn}}^2}=\expect{\condexpect{f(X_{i,\rbM})}{X_{\tilde{\Nn},\tilde{\Nn} }}^2}
\end{equation}
 in order to prove
\eqref{tail_prop}.
Both $\Nn$ and $\tilde{\Nn}$ are countably infinite sets
and $\rbM \cap \tilde{\rbM}=\emptyset$, so there is a bijection
between $\Nn$ and $\tilde{\Nn}$ which fixes $M$. Thus by
exchangeability we have
$\left(X_{\Nn,\Nn},X_{i,\rbM}\right) \sim \left(X_{\tilde{\Nn} ,\tilde{\Nn} }, X_{i,\rbM} \right)$, which implies
$\condexpect{f(X_{i,\rbM})}{X_{\Nn,\Nn}} \sim \condexpect{f(X_{i,\rbM})}{X_{\tilde{\Nn},\tilde{\Nn}}}$
 from which \eqref{tail_prop} and \eqref{equal_norm}   follow.
 Thus
\begin{multline} \label{condexp_4}
\expect{\condexpect{f(X_{i,\rbM})}{X_{\Nn,\Nn}}g(X_{\tilde{\rbM},\rbM})h(X_{\rbM,\rbM})}=\\
\expect{\condexpect{f(X_{i,\rbM})}{X_{\tilde{\Nn},\tilde{\Nn}}}g(X_{\tilde{\rbM},\rbM})h(X_{\rbM,\rbM})}
\end{multline}

Now we apply the permutation $\tau$ again. Using exchangeability we get
 \begin{multline} \label{condexp_5}
\expect{\condexpect{f(X_{i,\rbM})}{X_{\tilde{\Nn},\tilde{\Nn}}}g(X_{\tilde{\rbM},\rbM})h(X_{\rbM,\rbM})}=\\
\expect{\condexpect{f(X_{i,\rbM})}{X_{\tilde{\Nn},\tilde{\Nn}}}g(X_{\tilde{\rbN},\rbM})h(X_{\rbM,\rbM})}.
 \end{multline}
putting together 
 \eqref{condexp_2}, \eqref{tail_prop},
\eqref{condexp_4} and \eqref{condexp_5} we get \eqref{condexp_1}.
\end{proof}

\begin{lemma}\label{edges_indep_lemma}
 $\left(X_{i,j} \right)_{1 \leq i \leq j}$ are conditionally
independent given $X_{\Nn,\Nn \N}$.
\end{lemma}
\begin{proof}
Let $n \in \N$  and $f_{i,j} \in
L^{\infty}(\N_0)$ for all $1 \leq i \leq j \leq n$. By \eqref{defining_eq_of_cond_independence} we need to show
that
\begin{equation}\label{cond_indep_2}
\condexpect{\prod_{1 \leq i \leq j \leq n} f_{i,j}(X_{i,j})
}{X_{\Nn, \Nn  \N}}= \prod_{1 \leq i \leq j \leq n}
\condexpect{f_{i,j}(X_{i,j}) }{X_{\Nn,\Nn  \N}}
\end{equation}

In fact if we show that for all $\rbN,\tilde{\rbN} \subseteq \N$,
 $\rbN \cap
\tilde{\rbN} =\emptyset$, $f \in L^{\infty}(\A_{\rbN, \rbN})$, $g
\in L^{\infty}(\A_{\tilde{\rbN}, \rbN  \tilde{\rbN}})$ we have
\begin{equation}\label{cond_indep_3}
\condexpect{f(X_{\rbN,\rbN}) g(X_{\tilde{\rbN},\rbN  \tilde{\rbN}})}{X_{\Nn,\Nn \N}}=
\condexpect{f(X_{\rbN,\rbN}) }{X_{\Nn,\Nn  \N}}
\condexpect{g(X_{\tilde{\rbN},\rbN  \tilde{\rbN}})}{X_{\Nn,\Nn
 \N}}
\end{equation}
then we are done: in order to prove \eqref{cond_indep_2}
 using
\eqref{cond_indep_3} we first remove the terms corresponding to
the diagonal elements. By repeatedly
applying \eqref{cond_indep_3} with the choice $\rbN=\{i\}$  and $\tilde{\rbN}=[n] \setminus i$
for all
the elements $1 \leq i \leq n$
 we get
\begin{multline}
\condexpect{\prod_{1 \leq i \leq j \leq n} f_{i,j}(X_{i,j})
}{X_{\Nn, \Nn  \N}}=\\
\left(\prod_{i=1}^n \condexpect{f_{i,i}(X_{i,i})
}{X_{\Nn,\Nn \N}}\right)\cdot\condexpect{\prod_{1\leq i<j\leq n } f_{i,j}(X_{i,j})  }{ X_{\Nn,\Nn \N}}
\end{multline}
Then we can factorize the product corresponding to the
non-diagonal part by applying \eqref{cond_indep_3} repeatedly
  with the choice $\rbN=\{ i,j\}$, $\tilde{\rbN}=[n] \setminus \{i,j\}$
and $f(X_{\rbN,\rbN}):=f_{i,j}(X_{i,j})$ for all $1 \leq i <j \leq n$.

By Lemma \ref{eztkellellenoriznifeltfuggetlenseghez} in order to prove \eqref{cond_indep_3} we only need to show that
for all finite $\rbN' \subseteq \N$, $\rbM \subseteq \Nn$,
 and $h \in L^{\infty}(\A_{\rbM, \rbM
\rbN'})$ we have
\begin{multline}
\expect{f(X_{\rbN,\rbN})g(X_{\tilde{\rbN},\rbN
\tilde{\rbN}})
h(X_{\rbM,\rbM  \rbN'})}=
\expect{\condexpect{f(X_{\rbN,\rbN})}{X_{\Nn,\Nn \N}}
g(X_{\tilde{\rbN},\rbN \tilde{\rbN}})h(X_{\rbM,\rbM
\rbN'})}
\end{multline}

$\rbN \cap \rbN'$ and $\tilde{\rbN} \cap \rbN'$ need not be empty. By
increasing the support of $f$ or $g$ we might assume that $\rbN'=\rbN
\cup \tilde{\rbN}$ with retaining the property $\rbN \cap
\tilde{\rbN}=\emptyset$. Thus we need to show
\begin{multline}\label{condexp11}
\expect{f(X_{\rbN,\rbN})g(X_{\tilde{\rbN},\rbN
\tilde{\rbN}})
h(X_{\rbM,\rbM \rbN \tilde{\rbN}})}=
\expect{\condexpect{f(X_{\rbN,\rbN})}{X_{\Nn,\Nn \N}}
g(X_{\tilde{\rbN},\rbN \tilde{\rbN}})h(X_{\rbM,\rbM \rbN \tilde{\rbN}})}
\end{multline}

Let $\tilde{\rbM} \subseteq \Nn$ be such that $\tilde{\rbM} \cap \rbM
=\emptyset$ and $\abs{\tilde{\rbM}}=\abs{\tilde{\rbN}}$. Denote by
$\tau$ the permutation that swaps the elements of $\tilde{\rbN}$ with
those of $\tilde{\rbM}$. Using exchangeability and the fact that $g(X_{\tilde{\rbM},\rbN \tilde{\rbM}})h(X_{\rbM,\rbM \rbN \tilde{\rbM}})$ is
$\sigma\left(X_{\Nn,\Nn \N}\right)$-measurable we get
\begin{multline}
\expect{f(X_{\rbN,\rbN})g(X_{\tilde{\rbN},\rbN
\tilde{\rbN}})
h(X_{\rbM,\rbM \rbN \tilde{\rbN}})}=
\expect{f(X_{\rbN,\rbN})g(X_{\tilde{\rbM},\rbN
\tilde{\rbM}})
h(X_{\rbM,\rbM \rbN \tilde{\rbM}})}=\\
\expect{
\condexpect{ f(X_{\rbN,\rbN}) }{ X_{\Nn,\Nn \N} } g(X_{\tilde{\rbM},\rbN \tilde{\rbM}})
h(X_{\rbM,\rbM \rbN \tilde{\rbM}})}
\end{multline}

Now both $\Nn$ and $\tilde{\Nn}:= \Nn \setminus \tilde{\rbM}$ are countably
infinite sets, so similarly to \eqref{tail_prop} we have
\[\condexpect{f(X_{\rbN,\rbN})}{X_{\Nn,\Nn \N}}=
\condexpect{f(X_{\rbN,\rbN})}{X_{\tilde{\Nn},\tilde{\Nn}
\N}}
\]

Thus by applying $\tau$ again and using exchangeability we get \eqref{condexp11} similarly to the final steps of the proof of Lemma \ref{rows_cond_indep_lemma}.
\end{proof}

\begin{lemma}\label{cond_indep_fapipa}
 If \emph{$\rbN \subseteq \N$}  then \emph{$X_{\rbN,\rbN}$} and $X_{\Nn, \Nn  \N}$
are conditionally independent given \emph{$X_{\Nn,\Nn \rbN}$}.
\end{lemma}
\begin{proof}
By Lemma \ref{lemma_vegesites} and Lemma \ref{eztkellellenoriznifeltfuggetlenseghez} we only have to show that for all $\rbM \subseteq \Nn$, $\tilde{\rbN} \subseteq \N$,
 $\tilde{\rbN} \cap \rbN=\emptyset$,
 $f \in L^{\infty}(\A_{\rbN, \rbN})$,
$g \in L^{\infty}(\A_{\rbM, \rbM \tilde{\rbN} \rbN  })$,
 $h \in L^{\infty}(\A_{\rbM , \rbM \rbN})$ we have
\begin{multline}\label{condexpect_13}
\expect{f(X_{\rbN,\rbN}) g(X_{\rbM,\rbM \tilde{\rbN} \rbN  })
h(X_{\rbM, \rbM \rbN})}=
\expect{ \condexpect{ f(X_{\rbN,\rbN})}{X_{\Nn,\Nn \rbN}}
 g(X_{\rbM,\rbM \tilde{\rbN} \rbN })
h(X_{\rbM, \rbM \rbN})}
\end{multline}

Choose $\tilde{\rbM} \subseteq \Nn$ such that
$\abs{\tilde{\rbM}}=\abs{\tilde{\rbN}}$ and
 $\tilde{\rbM} \cap \rbM  =\emptyset$.
  By applying a permutation that swaps the elements of $\tilde{\rbN}$ and $\tilde{\rbM}$, using exchangeability
  and the fact that $g(X_{\rbM,\rbM  \tilde{\rbM} \rbN })h(X_{\rbM,
\rbM \rbN})$ is $\sigma \left( X_{\Nn,\Nn \rbN} \right)$-measurable
     we get
\begin{multline}
\expect{f(X_{\rbN,\rbN}) g(X_{\rbM,\rbM  \tilde{\rbN} \rbN })
h(X_{\rbM, \rbM \rbN})}=
\expect{f(X_{\rbN,\rbN}) g(X_{\rbM,\rbM  \tilde{\rbM} \rbN })
h(X_{\rbM, \rbM \rbN})}=\\
\expect{ \condexpect{f(X_{\rbN,\rbN})}{X_{\Nn,\Nn \rbN}}
g(X_{\rbM,\rbM  \tilde{\rbM} \rbN }) h(X_{\rbM,
\rbM \rbN})}
\end{multline}
Let $\tilde{\Nn}=\Nn \setminus \tilde{\rbM}$. Similarly to \eqref{tail_prop} we get
\[\condexpect{f(X_{\rbN,\rbN})}{X_{\Nn,\Nn \rbN}} = \condexpect{f(X_{\rbN,\rbN})}{X_{\tilde{\Nn},\tilde{\Nn} \rbN}},\] from which
\eqref{condexpect_13} follows in a similar fashion as in the previous two lemmas.
\end{proof}

$ $

\begin{proof}[Proof of Theorem \ref{theorem_aldous} \eqref{theorem_aldous_negyvaltozos}]
Our proof follows \cite{aldous_exch}.
We consider the extended process $X_{\Nn  \N, \Nn \N}$.

 Let $\alpha$, $\left(U_{i}\right)_{i \in \N}$,
$\left(\beta_{i,j} \right)_{1 \leq i\leq j }$ be jointly i.i.d. $U
\lbrack 0, 1 \rbrack$ random variables independent from $X_{\Nn  \N, \Nn \N}$.

Pick $i_0, j_0 \in \N$, $i_0 < j_0$.
 By Lemma \ref{codinglemma} \eqref{codinglemma3} there exists a
 \[h: \A_{\Nn,\Nn} \times \N_0^{\Nn} \times \N_0^{\Nn} \times [ 0, 1
 ] \to \N_0
  \] such
that $h(z,x_1,x_2,b)\equiv h(z,x_2,x_1,b)$ and
 putting
\begin{equation*}
 X^*_{i_0,j_0}=h \big(X_{\Nn,\Nn}, X_{i_0,\Nn}, X_{j_0,\Nn},
\beta_{i_0,j_0}\big)
\end{equation*}
 we have
\begin{equation} \label{h_distribution_property}
\big(X_{\Nn,\Nn}, X_{i_0,\Nn}, X_{j_0,\Nn},
X^*_{i_0,j_0} \big) \sim \big(X_{\Nn,\Nn}, X_{i_0,\Nn},
X_{j_0,\Nn}, X_{i_0,j_0} \big).
\end{equation}
  For $i \neq j \in \N$ we define
\begin{equation}\label{maszlag}
 X^*_{i,j}:=h\big(
X_{\Nn,\Nn},X_{i,\Nn}, X_{j,\Nn}, \beta_{\min(i,j) , \max(i,j)} \big)
\end{equation}

 By Lemma \ref{codinglemma} \eqref{codinglemma2}
there exists a $\bar{h}$ such that putting
\[ X^*_{i_0,i_0}=\bar{h} \big(X_{\Nn,\Nn}, X_{i_0,\Nn},
\beta_{i_0,i_0}\big) \] we have
\[ \big(X_{\Nn,\Nn}, X_{i_0,\Nn},
X^*_{i_0,i_0} \big) \sim \big(X_{\Nn,\Nn}, X_{i_0,\Nn},
X_{i_0,i_0} \big)\]  For $i \in \N$, we define
\begin{equation}
\label{maszlag2} X^*_{i,i}:=\bar{h}\big(
X_{\Nn,\Nn},X_{i,\Nn}, \beta_{i,i} \big)
\end{equation}

Having defined $X^*_{i,j}$ for all $i,j \in \N$ we use
Lemma \ref{identificationlemma} to prove
\begin{equation}\label{futunk}
(X_{\Nn ,\Nn \N},  X_{\N,\N}) \sim
(X_{\Nn ,\Nn \N}, X^*_{\N,\N})
\end{equation}

 $\left(X_{i,j} \right)_{1\leq i \leq j}$ are
conditionally independent given $X_{\Nn,\Nn \N}$ by Lemma
\ref{edges_indep_lemma}.  $\left(X^*_{i,j} \right)_{1\leq i \leq j}$ are
conditionally independent given $X_{\Nn,\Nn \N}$ since $\left(\beta_{i,j} \right)_{1\leq i \leq j}$ are
independent from each other and $X_{\Nn \N, \Nn \N}$. Thus we only need to prove that for each $i,j \in
\N$ the random variables $X_{i,j}$ and $X^*_{i,j}$ are
conditionally identically distributed given $X_{\Nn, \Nn \N}$.
 We prove this when $i \neq j$, the
proof of the diagonal case is similar.

Let $\tau$ denote a
permutation of $\N$ such that $\tau(\{i_0,j_0\})=\{i,j\}$. By
relabeling (and using $X_{i,j}= X_{j,i}$, $X^*_{i,j}=X^*_{j,i}$) we might assume that $\tau(i_0)=\tau(i)$ and
$\tau(j_0)=\tau(j)$ and that $\tau=\tau^{-1}$. We extend
$\tau$ to $\Nn \cup \N$ in such a way that $\forall k \in \Nn \;
\tau(k)=k$.
\begin{align*}
&\big( X_{\Nn,\Nn}, X_{i,\Nn}, X_{j,\Nn},
X^*_{i,j} \big) \stackrel{\eqref{maszlag}}{=} \\
&\big( X_{\Nn,\Nn}, X_{i,\Nn}, X_{j,\Nn},
h\big(
X_{\Nn,\Nn},X_{i,\Nn}, X_{j,\Nn}, \beta_{\min(i,j) , \max(i,j)} \big) \big) \stackrel{\eqref{vertex_exch}}{\sim} \\
&\big( X_{\Nn,\Nn}, X_{i_0,\Nn}, X_{j_0,\Nn},
h\big(
X_{\Nn,\Nn},X_{i_0,\Nn}, X_{j_0,\Nn}, \beta_{\min(i,j) , \max(i,j)} \big)
 \big) \sim \\
 &\big( X_{\Nn,\Nn}, X_{i_0,\Nn}, X_{j_0,\Nn},
h\big(
X_{\Nn,\Nn},X_{i_0,\Nn}, X_{j_0,\Nn}, \beta_{i_0 , j_0} \big)
 \big) \stackrel{\eqref{maszlag}}{=} \\
 &\big( X_{\Nn,\Nn}, X_{i_0,\Nn}, X_{j_0,\Nn},
X^*_{i_0,j_0}
 \big) \stackrel{\eqref{h_distribution_property}}{\sim} \\
 &\big( X_{\Nn,\Nn}, X_{i_0,\Nn}, X_{j_0,\Nn},
X_{i_0,j_0}
 \big) \stackrel{\eqref{vertex_exch}}{\sim} \\
& \big( X_{\Nn,\Nn}, X_{i,\Nn}, X_{j,\Nn},
X_{i,j} \big)
\end{align*}

As a consequence we get for all $f \in L^{\infty}(\N_0)$
\begin{equation}\label{michaeljacksonisdead}
\condexpect{f(X_{i,j})}{X_{\Nn,\Nn \rbN}}=\condexpect{f(X^*_{i,j})}{X_{\Nn,\Nn \rbN}}
\end{equation}
with $\rbN=\{i,j\}$.
Now using Lemmas \ref{cond_indep_drop_lemma}, \ref{cond_indep_fapipa}
 and the fact that $\beta_{\min(i,j),\max(i,j)}$ is independent
form $X_{\Nn \N, \Nn \N}$  we get
\begin{multline*}
\condexpect{f(X_{i,j})}{X_{\Nn, \Nn \N}}=
\condexpect{f(X_{i,j})}{X_{\Nn, \Nn \N}, X_{\Nn, \Nn \rbN} } \stackrel{\eqref{cond_indep_drop_formula}}{=}
\condexpect{f(X_{i,j})}{X_{\Nn,\Nn \rbN}}\stackrel{\eqref{michaeljacksonisdead}}{=} \\
\condexpect{f(X^*_{i,j})}{X_{\Nn,\Nn \rbN}}\stackrel{\eqref{cond_indep_drop_formula}}{=}
\condexpect{f(X^*_{i,j})}{X_{\Nn, \Nn \N}, X_{\Nn, \Nn \rbN} }=
\condexpect{f(X^*_{i,j})}{X_{\Nn, \Nn \N}}
\end{multline*}
which proves that $X_{i,j}$ and $X^*_{i,j}$ are
conditionally identically distributed given $X_{\Nn, \Nn \N}$.

Having established \eqref{futunk} the next
step is to code $X_{i,\Nn}$ for each $i \in \N$. By Lemma \ref{codinglemma} \eqref{codinglemma2}
 there exists $g$ such that, putting
$X^*_{i_0,\Nn}:=g(X_{\Nn,\Nn}, U_{i_0})$ we have
\begin{equation*}
 \big(X_{\Nn,\Nn},X^*_{i_0,\Nn} \big) \sim \big(X_{\Nn,\Nn},X_{i_0,\Nn} \big)
\end{equation*}
For each $i \in \N$ define
\begin{equation}
 \label{tralala}
X^*_{i,\Nn}:=g(X_{\Nn,\Nn}, U_i)
\end{equation}
 Then
\[\big(X_{\Nn,\Nn},X^*_{i,\Nn} \big) \sim
\big(X_{\Nn,\Nn},X_{i,\Nn} \big)\] since neither distribution
depends on $i$.

 We now want to use Lemma \ref{identificationlemma} to prove
 \begin{equation} \label{butyok}
 \big(X_{\Nn,\Nn}, X^*_{i,\Nn} \big)_{i \in \N} \sim
\big(X_{\Nn,\Nn},X_{i,\Nn} \big)_{i \in \N}.
\end{equation}
$\big(X^*_{i,\Nn} \big)_{i \in \N}$ are conditionally
independent given $X_{\Nn,\Nn}$ by since $\left(U_i\right)_{i \in \N}$ are independent from each other and
$X_{\Nn \N, \Nn \N}$.
$\big(X_{i,\Nn} \big)_{i \in \N}$ are conditionally independent
given $X_{\Nn,\Nn}$ by Lemma \ref{rows_cond_indep_lemma}, thus
\eqref{butyok}.

Putting together  \eqref{maszlag}, \eqref{maszlag2}, \eqref{futunk},\eqref{tralala} and  \eqref{butyok}  we see
that $(X_{\Nn,\Nn}, X_{\N,\N})\sim (X_{\Nn,\Nn}, \ddot{X}_{\N,\N})$ where
\begin{equation}\label{ddX}
\ddot{X}_{i,j}:=
\begin{cases}
h \big(X_{\Nn,\Nn}, g(X_{\Nn,\Nn}, U_i), g(X_{\Nn,\Nn},
U_j), \beta_{\min(i,j),\max(i,j)} \big) & \text{if } i\neq j \\
\bar{h}\big( X_{\Nn,\Nn}, g(X_{\Nn,\Nn}, U_i),
\beta_{i,i} \big) & \text{if } i= j
\end{cases}
\end{equation}
Finally, code $X_{\Nn,\Nn}$ as a function of $\alpha$ using Lemma \ref{codinglemma} \eqref{codinglemma1} and the
desired representation \eqref{representation_formula_f} is established.
\end{proof}

\subsection{Dissociated arrays}

Now we show how a special form \eqref{representof_exch_dissoc}
 of the general representation \eqref{representation_formula_f}
correspond to the dissociated property (Definition \ref{def_dissociated}) of the vertex
exchangeable array
 $X_{\N,\N}$.

 First we define various $\sigma$-algebras on the measurable space $\left(\A_\N, \mathcal{F} \right)$.

   Let ${\mathcal{S}_n\subset \mathcal{F}}$ denote the $\sigma$-algebra generated by the events that are invariant under the relabeling of the first $n$ nodes. So an $\mathcal{F}$-measurable function $g: \A_{\N} \to \R$ is $\mathcal{S}_n$-measurable iff it satisfies
   \begin{equation}\label{symmetric_function_formula}
 \forall\;  A\in \A_{\N}:\, g\left( (A(i,j))_{i,j=1}^{\infty}\right) = g\left( (A(\tau(i),\tau(j)))_{i,j=1}^{\infty}\right)
   \end{equation}
     for every permutation $\tau :\, \N\to\N$  satisfying $\tau(n')=n'$ for every $n'>n$. Clearly $\mathcal{S}_n\supseteq \mathcal{S}_{n+1}$. Define the \emph{ $\sigma$-algebra of symmetric events} $\mathcal{S}_{\infty}$ by
    \begin{equation}\label{def_symmetric_sigma_algebra}
\mathcal{S}_{\infty}:=\bigcap_{n=1}^{\infty} \mathcal{S}_n
    \end{equation}
  Alternatively, a function $g$ is $\mathcal{S}_{\infty}$-measurable iff \eqref{symmetric_function_formula} holds for any finitely supported permutation $\tau$. Define the $\sigma$-algebras
  \begin{equation}\label{FG_sigma_algebras}
  \mathcal{F}_n:=\sigma \left( \big(A(i,j)\big)_{i,j=1}^n \right) \qquad \mathcal{G}_n:=
   \sigma \left( \big( A(i,j) \big)_{i,j=n+1}^{\infty} \right)
  \end{equation}
  $\mathcal{F}_1 \subseteq \mathcal{F}_2 \subseteq \dots$ and $\mathcal{G}_1 \supseteq \mathcal{G}_2 \supseteq \dots$.
  The \emph{ $\sigma$-algebra of tail events} $\mathcal{G}_{\infty}$ is defined by
 \begin{equation}\label{def_tail_sigma_algebra}
\mathcal{G}_{\infty}:=\bigcap_{n=1}^{\infty} \mathcal{G}_n
    \end{equation}
  It is easy to see that $\mathcal{G}_n \subseteq \mathcal{S}_n$, which implies $\mathcal{G}_{\infty} \subseteq \mathcal{S}_{\infty}$.

\begin{theorem}\label{theorem_dissociated_equivalent}
The following properties are equivalent for a vertex exchangeable
array $X_{\N,\N}$:
\begin{enumerate}
\item \label{sym_triv}  $\mathcal{S}_{\infty}$ is trivial, i.e. it contains only events with probability $0$ or $1$.
\item \label{tail_triv}  $\mathcal{G}_{\infty}$ is trivial, i.e. it contains only events with probability $0$ or $1$.
\item \label{dissociated} $X_{\N,\N}$ is dissociated.
\item \label{no_alpha} There exists a measurable function $g:[0,1]^3 \to \N_0$
  such that
  \begin{equation}\label{g_no_alpha_symmetric}
  g(u_1,u_2,b)\equiv g(u_2,u_1,b)
  \end{equation}
   and defining
\begin{equation}\label{no_alpha_formula}
\tilde{X}_{i,j}=g \big( U_i,U_j,\beta_{\min(i,j) ,\max(i,j)} \big)
\end{equation}
we have $\tilde{\mathbf{X}} \sim \mathbf{X}$.
\end{enumerate}
\end{theorem}
The equivalence of \ref{sym_triv}. and \ref{dissociated}. is proved in \cite{Lovasz_Szegedy_2009}, Proposition 3.6 using
a different terminology (and proof): a consistent countable
random graph model is \emph{local} if and only if it is \emph{ergodic}.

\begin{proof}[ Proof of \ref{sym_triv}. $\iff$ \ref{tail_triv} ]

$ $

In fact we prove that  $\mathcal{S}_{\infty}$ and  $\mathcal{G}_{\infty}$ are
essentially the same. This is a version of the Hewitt-Savage
theorem (\cite{shiryaev}, page 382).
 The inclusion $\cG_{\infty}
\subseteq \mathcal{S}_{\infty}$ is obvious. Conversely we show that any
bounded  $\cS_{\infty}$-measurable random variable is essentially
$\cG_{\infty}$-measurable. It is enough to show that for any $f \in
L^2 \left( \A_{\N}, \mathcal{\cS_{\infty}}, \probp_X \right)$  we have
\begin{equation}\label{symmetric_tail_are_the_same_formula}
 \condexpect{f \left(X_{\N,\N} \right)}{\cG_{\infty}}=f \left(X_{\N,\N} \right)
 \end{equation}
Recall the definitions of the $\sigma$-algebras $\cF_n$ and $\cG_n$ in \eqref{FG_sigma_algebras}.
In order to prove \eqref{symmetric_tail_are_the_same_formula} it is enough to show that
\begin{equation}\label{tralala13}
 \forall \, \varepsilon>0\;,\forall \, m  \quad
\expect{ \left( \condexpect{f(X_{\N,\N})}{ \cG_m} -f(X_{\N,\N}) \right)^2 } <\varepsilon.
\end{equation}

Fix $\varepsilon>0$ and $m$. It follows
from \eqref{vegesites_eq} and \eqref{upward}
 that there exists an $n \in \N$
such that defining $\rbN:=[n]$ we have
\[\expect{ \left(
\condexpect{f(X_{\N,\N})}{X_{\rbN, \rbN}} -f(X_{\N,\N})
\right)^2 } <\varepsilon \] We might assume $m \leq n$.
Thus by \eqref{measurability_trick} there is a function $g: \A_n \to \R$ such that
\[\expect{ \left(
g(X_{\rbN, \rbN}) -f(X_{\N,\N}) \right)^2 } <\varepsilon.\]
 Now take $\tilde{\rbN} \subseteq \N$, $\abs{\tilde{\rbN}}=\abs{\rbN}=n$, $\tilde{\rbN} \cap \rbN=\emptyset$
 and a permutation $\tau$ that swaps the elements of $\rbN$ with those of $\tilde{\rbN}$.
Since $f$ is $\cS_{\infty}$-measurable we have
$f\left( \left( X_{i,j} \right)_{i,j=1}^{\infty} \right)=  f\left( \left( X_{\tau(i),\tau(j)} \right)_{i,j=1}^{\infty} \right)$.
If we combine this with the vertex-exchangeability of $X_{\N,\N}$  we get
\[\expect{ \left(
g(X_{\tilde{\rbN}, \tilde{\rbN}}) -f(X_{\N,\N}) \right)^2 } <\varepsilon.\]
Now $\tilde{\rbN} \subseteq \{m+1,m+2,\dots \}$, thus  $g(X_{\tilde{\rbN}, \tilde{\rbN}})$
is $\cG_m$-measurable,
and by the conditional version of Steiner's inequality \eqref{steiner} we get
\begin{equation*}
\expect{ \left( \condexpect{f(X_{\N,\N})}{\cG_m} -f(X_{\N,\N}) \right)^2 } \leq \expect{
\left( g(X_{\tilde{\rbN}, \tilde{\rbN}}) -f(X_{\N,\N}) \right)^2 }
<\varepsilon.
\end{equation*}
Thus \eqref{tralala13} is established.
\end{proof}

\begin{proof}[Proof of \ref{tail_triv}. $\iff$ \ref{dissociated}]

$ $

This proof is based on the proof of Theorem 5.5 in
\cite{diaconis_janson}.

First we prove   \ref{tail_triv}. $\implies$ \ref{dissociated}.
 In order to prove that $X_{\N,\N}$ is dissociated we only need to show
  that for all $\varepsilon>0$, for all finite subsets $\rbN, \rbM \subseteq \N$  such that
   $\rbM \cap \rbN   =\emptyset$ and for all $f \in L^{\infty}(\A_{\rbM, \rbM})$ and
$g \in L^{\infty}(\A_{\rbN, \rbN})$ we have
\begin{equation}\label{tralala14}
\left| \expect{
f(X_{\rbM,\rbM})g(X_{\rbN,\rbN})}- \expect{
f(X_{\rbM,\rbM})} \expect{g(X_{\rbN,\rbN})} \right| \leq \varepsilon
\end{equation}
We might assume $\|g\|_{\infty}\leq 1$. Since $\cG_{\infty}$ is
trivial, by \eqref{downward} there is an $n \in \N$ such that
\begin{equation}\label{badacsonyors}
 \expect{ \left|
 \condexpect{f(X_{\rbM,\rbM})}{ \cG_n}-\expect{f(X_{\rbM,\rbM})}
    \right| } \leq \varepsilon
    \end{equation}
We might assume $\rbM \cup \rbN \subseteq [n]$.
Let $\tilde{\rbN} \subseteq \{n+1,n+2,\dots \}$,
 $\abs{ \tilde{\rbN}}=\abs{\rbN}$. By exchangeability we have
\begin{multline*}
\expect{ f(X_{\rbM,\rbM})g(X_{\rbN,\rbN})}- \expect{
f(X_{\rbM,\rbM})} \expect{g(X_{\rbN,\rbN})}=\\
\expect{ f(X_{\rbM,\rbM})g(X_{\tilde{\rbN},\tilde{\rbN}})}- \expect{
f(X_{\rbM,\rbM})} \expect{g(X_{\tilde{\rbN},\tilde{\rbN}})}=\\
\expect{\left[\condexpect{f(X_{\rbM,\rbM})}{\cG_n}-\expect{f(X_{\rbM,\rbM})}\right]g(X_{\tilde{\rbN},\tilde{\rbN}})}
\end{multline*}
Now \eqref{tralala14} follows from $\|g\|_{\infty}\leq 1$ and \eqref{badacsonyors}.

The proof of \ref{dissociated}. $\implies$  \ref{tail_triv}. is similar to that of Kolomorov's 0-1 law (\cite{williams}, Theorem 4.11):
 From the dissociated property of $\probp_X$ it follows that for all $n \in \N$ the $\sigma$-algebras
$\cF_n$ and $\cG_n$ are independent. From standard approximation arguments we get that $\cF=\sigma\left(\bigcup_{n=1}^{\infty} \cF_n\right)$
and $\cG_{\infty}=\bigcap_{n=1}^{\infty} \cG_n$ are also independent, thus $\cG_{\infty}$ is independent from itself, thus it can only contain
events of probability $0$ or $1$.
\end{proof}

\begin{proof}[ Proof of \ref{dissociated}. $\iff$ \ref{no_alpha}]

$ $

This proof is based on the proof of Proposition 3.3 in
\cite{aldous_exch}.

The proof of \ref{no_alpha}. $\implies$ \ref{dissociated}. is trivial since an array $\tilde{\mathbf{X}}$
of form \eqref{no_alpha_formula} is dissociated.

 As for
the other direction: It is easy to see that the extended array
$X_{ \Nn \N, \Nn \N}$ is also dissociated from which the
independence of $X_{\Nn,\Nn}$ and $X_{\N,\N}$ follows.

The construction in the proof of Theorem \ref{theorem_aldous}
(see \eqref{ddX}) gave an array $\ddot{X}_{\N,\N}$ of the form
\[\ddot{X}_{i,j}=\hat{f}(X_{\Nn,\Nn},
U_i,U_j,\beta_{\min(i,j),\max(i,j)} )\]
 which was shown to satisfy
$ \big( X_{\Nn,\Nn}, X_{\N,\N}\big) \sim \big( X_{\Nn,\Nn},
\ddot{X}_{\N,\N} \big)$.
 But $X_{\Nn,\Nn}$ and $X_{\N,\N}$
are independent, thus the final step of the proof of Theorem
\ref{theorem_aldous} (coding $X_{\Nn,\Nn}$ as a function of
$\alpha$) gives a representation \eqref{representation_formula_f} in
which $\tilde{X}_{\N,\N}$ is independent of $\alpha$.
Thus it holds for almost any $a \in [0,1]$ that the random variables
$\left( f(a,U_i,U_j, \beta_{\min(i,j),\max(i,j)}) \right)_{i,j=1}^{\infty}$ have the same distribution as $X_{\N,\N}$.
 Defining \[g(u_1,u_2,b):=f(a,u_1,u_2,b)\] for any such $a$ gives the desired representation \eqref{no_alpha_formula}.

\end{proof}

\section{Proof of Theorem \ref{theorem_main_result}}\label{section_proofs}

\subsection{Proof of Theorem \ref{theorem_main_result}: $(a)\Rightarrow (b)$}
The idea that Theorem \ref{theorem_aldous} can be applied to prove $(a)\Rightarrow (b)$ comes from \cite{diaconis_janson}.

  If $f\in \T$, then by Lemma \ref{lemma_graphkonv_konvindist} there exists a multigraph sequence $\left( G_n \right)_{n=1}^{\infty}$ and
  a random array $\mathbf{X}$
such that $\rndarray{G_n}{} \toindis \rndarray{}{}$ and $\prob{\forall i,j\leq k: A(i,j) \leq X(i,j)} = f(A)$ holds.
    \rndarray{G_n}{} is vertex exchangeable and dissociated for all $n$ so its limit \rndarray{}{} is also vertex exchangeable and dissociated.
Thus by the \ref{dissociated}. $\implies$ \ref{no_alpha}. implication of Theorem \ref{theorem_dissociated_equivalent}  we get that
$\mathbf{X} \sim \tilde{\mathbf{X}}$ where $\tilde{\mathbf{X}}$ is defined by \eqref{no_alpha_formula}. Define the multigraphon $W$ by
\[ W(x,y,k):=\int_0^1 \ind[ g(x,y,b)=k] \, \text{d} b =\prob{g(x,y,\beta)=k} \]
where $\beta \sim U[0,1]$.  $W$ is indeed a multigraphon: \eqref{W_symetric_eq} follows from \eqref{g_no_alpha_symmetric}, and
\eqref{W_nondefective_eq} and \eqref{W_diag_even} follow from the fact that $\mathbf{X}$ is a random element of $\A_{\N}$ (see \eqref{A_N_def_formula}).
From \eqref{no_alpha_formula}, Definition \ref{def_X_W} and Lemma \ref{identificationlemma} it follows that
 $\left( \left(U_i \right)_{i \in \N},  \tilde{\mathbf{X}} \right) \sim
 \left( \left(U_i \right)_{i \in \N},  \mathbf{X}_W \right)$.
      Thus for any $F \in \M$ with adjacency matrix $A$ we have
\begin{multline*}
 f(A)=\prob{\forall i,j\leq k: A(i,j) \leq \tilde{X}(i,j)}=\\ \prob{\forall i,j\leq k: A(i,j) \leq X_W(i,j)} \stackrel{\eqref{homsur_grafon_valszamosan}}{=}  \homsur{F}{W}
 \end{multline*}

\subsection{Proof of Theorem \ref{theorem_main_result}: $(b)\Rightarrow (c)$}

Suppose that there exists a multigraphon $W$ for which $f(\,\cdot\,)=\homsur{\,\cdot\,}{W}$.
 It is easy to  check that the graph parameter \homsur{\,\cdot\,}{W} is normalized, multiplicative and non-defective.

  In order to prove that $f$ is reflection positive we only need to show that the connection matrices $M(k,f)$ are positive semidefinite for each $k\geq 0$. Let us fix $k$. Suppose that the finite minor of $M(k,f)$ is indexed by the $k$-labeled multigraphs $F_1, F_2, \dots, F_N$. We label the unlabeled vertices of $F_1, \dots, F_N$ using disjoint subsets of $\N$, so that the adjacency matrix $A_m \in \A_{V(F_m)}$  and for all $l,m \in [N]$, $l\neq m$ we have $V(F_l)\cap V(F_m)=[k]$.

Recalling Definition \ref{def_X_W} we define the random variables $Y_1,\dots,Y_N$ by
 \[Y_m := \ind \left[ \forall \, i,j \in V(F_m)  :  A_m(i,j) \leq X_W(i,j)   \right]. \]
\begin{multline*}
\left( M(k,f)\right)(F_l, F_m) = f(F_lF_m)
  = \homsur{F_l F_m}{W}
   \stackrel{\eqref{adjmatrix_of_F1F2}}{=} \homsur{A_l \vee A_m}{W} \stackrel{ \eqref{homsur_grafon_valszamosan}}{=}\\
    \expect{ \ind \left[ \forall \, i,j \in V(F_l)\cup V(F_m) :  (A_l \vee A_m)(i,j) \leq X_W(i,j)  \right] } =
    \expect{Y_l\cdot Y_m}
\end{multline*}
The finite minors of $M(k,f)$ are positive semidefinite because for any $v:[N]\rightarrow \R$ we have
\begin{equation*}
  v^T\cdot M(k,f)\cdot v = \sum_{l,m=1}^N \expect{Y_l\cdot Y_m}\cdot v_l v_m = \expect{\left( \sum_{m=1}^N v_m\cdot Y_m\right)^2} \geq 0.
\end{equation*}

\subsection{Proof of Theorem \ref{theorem_main_result}: $(c)\Rightarrow (d)$}

The proof follows the main idea of \cite{Lovasz_Szegedy_2006}.
%First we show that if $f$ is multipl icative than $f^{\dagger}$ is also multip licative.
%We need to show that $\forall F,G \in \M$
%we have $f^{\dagger}(F G) = f^{\dagger} (F) f^{\dagger} (G)$.
% We label the vertices of $F$, $G$ so that their adjacency matrices $A$, $B$  are
% indexed by $\rbM \cap \rbM = \emptyset$.
%\begin{multline*}
%f^{\dagger}(F G) \stackrel{\eqref{adjmatrix_of_F1F2}}{=} f^{\dagger}(A \vee B) \stackrel{\eqref{mobius_trans}}{=}
%\sum_{E \in \E_{MN} } (-1)^{e(E)}f( (A \vee B) +E)=\\
%\sum_{E_F \in \E_{M} } \sum_{E_G \in \E_{N} }
%(-1)^{e(E_F \vee E_G)}  f( (A \vee B) + (E_F \vee E_G ))=\\
%\sum_{E_F \in \E_{M} } \sum_{E_G \in \E_{N} }
%(-1)^{e(E_F)} (-1)^{e(E_G)}  f( (A+E_F) \vee (B+ E_G))=\\
%\sum_{E_F \in \E_{M} } \sum_{E_G \in \E_{N} }
%(-1)^{e(E_F)} (-1)^{e(E_G)}  f( A+E_F) f( B+ E_G) \stackrel{\eqref{mobius_trans}}{=} f^{\dagger} (F) f^{\dagger} (G)
%\end{multline*}
%It easily follows from \eqref{mobius_trans} that
%\[ \forall\, A_1, A_2 \in \A_k \; f(A_1 A_2)=f(A_1)f(A_2) \quad \implies \quad
%\forall\, A_1, A_2 \in \A_k \; f^{\dagger}(A_1 A_2)=f^{\dagger}(A_1)f^{\dagger}(A_2)\]
% Thus $f^{\dagger}$ is multip licative.
 We prove that if $f$ is reflection positive then $f^{\dagger} \geq 0$.

 Let us fix $k\in \N$. Let $F_1$ and $F_2$ be $k$-labeled graphs with no unlabeled nodes (thus $\abs{V(F_1)}=\abs{V(F_2)}=k$). By \eqref{adjmatrix_of_F1F2} the adjacency matrix of $F_1F_2$ is
 \[ (A_1\vee A_2)(i,j)= \max \left( A_1(i,j), A_2(i,j) \right) \]
    We introduce matrices whose rows and columns are indexed by adjacency matrices from $\A_k$. Define $M, Z, D\in \R^{\A_k \times\A_k}$ by
\begin{align*}
M(A_1, A_2) &:= f(A_1\vee A_2) \\
Z(A_1, A_2) &:= \indd{A_1 \leq A_2} \\
D(A_1, A_2) &:= \indd{A_1 = A_2}\cdot f^{\dagger}(A_1)
\end{align*}
It follows from $(c)$ that $M$ is positive semidefinite since it is a symmetric minor of the connection matrix $M(k,f)$.
 We show that $M=ZDZ^T$:
\begin{align*}
\left(ZDZ^T\right)(A,B) &= \sum\limits_{E,F\in \A_k}\!\! Z(A,E)\cdot D(E,F)\cdot Z^T(F,B) \\
%  &= \sum\limits_{E,F} \indd{A\leq E}\indd{E=F} f^{\dagger}(E) \indd{B\leq F} \\
  &= \sum\limits_{E \in \A_k} \indd{A\leq E} f^{\dagger}(E) \indd{B\leq E}
  = \!\! \sum\limits_{A\vee B\leq E} f^{\dagger}(E)  \\
  &= \left( f^{\dagger}\right)^{-\dagger}(A\vee B) \stackrel{(\ast)}{=} f(A\vee B) = M(A,B)
\end{align*}
The equation $(\ast)$ follows from the assumption $f(\infty)=0$ and \eqref{mobius_inverzios_tetel_eq}.

In order to prove that $f^{\dagger} \geq 0$
it suffices to show that $Z$ is invertible because then the diagonal matrix $D=Z^{-1}M\left( Z^{-1}\right)^T$ is also positive semidefinite so $f^{\dagger}(A)\geq 0$ for all $A\in\A_k$. It easily follows from \eqref{kiejti_kiveve_ha_egyezik} that the inverse matrix of $Z$ is
\begin{equation*}%\label{eq_Z^-1}
Z^{-1}(A_1, A_2)= (-1)^{e(A_2-A_1)} \indd{A_2-A_1 \in \E_k}   .
\end{equation*}

\subsection{Proof of Theorem \ref{theorem_main_result}: $(d)\Rightarrow (a)$}\label{subsection_d->a}
Assume given a normalized, multiplicative, non-defective multigraph parameter $f$ with non-negatine Möbius transform.
Our aim is to prove that $f \in \T$, that is \eqref{def_of_T_eq} holds. Our proof is the multigraph analogue of Corollary 2.6 of \cite{Lovasz_Szegedy_2006}, but we use reverse martingales instead of Azuma's inequality to prove \eqref{t<=_f_as}, as suggested in \cite{diaconis_janson}.

 First we define a probability measure $\probp_n$ on $\A_n$. For every $A\in \A_n: \, \probp_n(A) := f^{\dagger}(A)$. $\probp_n$ is indeed a probability measure since $f^{\dagger}\geq 0$ and using that $f$ is non-defective, multiplicative and normalized we get
\begin{equation*}
\sum_{A\in \A_n} f^{\dagger}(A) = \left( f^{\dagger}\right)^{-\dagger}(\zero_n) \stackrel{\eqref{mobius_inverzios_tetel_eq}}{=}
 f(\zero_n) = f(\zero_1)^n = 1.
\end{equation*}

Next we show that \eqref{kolmogorov_consistency_condition} holds, i.e. if \randarray{n}{}{n} has distribution $\probp_n$ and \randarray{n+1}{}{n+1} has distribution $\probp_{n+1}$ then $\randarray{n}{}{n} \sim \randarray{n+1}{}{n}$. We have to show that if we define $g: \A_n \to \R$ by
\begin{equation}\label{def_g_vonat}
 g(A):= \sum_{A' \in \A_{n+1}} \ind \left[ \left(A'\right)_{i,j=1}^n = A \right]\cdot f^{\dagger}(A'),
 \end{equation}
then $\forall \, A \in \A_n\; f^{\dagger}(A)=g(A)$. We only need to check that
$\forall \, A \in \A_n \; f(A)=g^{-\dagger}(A)$, because
\[ f \equiv g^{-\dagger} \quad \implies f^{\dagger} \equiv \left(g^{-\dagger}\right)^{\dagger} \stackrel{(\ast)}{\equiv} g \]
The proof of the identity $(\ast)$ is similar to that of Lemma \ref{mobius_inverzio_tetel}.

If $A\in \A_n$ then define $A_0\in \A_{n+1}$ by
\begin{equation*}
  A_0(i,j) =
  \begin{cases}
    A(i,j)    & \text{if $i,j\leq n$} \\
    0         & \text{if $\max(i,j)=n+1$}
  \end{cases}
\end{equation*}
(i.e., we add an isolated vertex labeled by $n+1$ to the labeled graph with adjacency matrix $A$).
Using that $f$ is non-defective, multiplicative and normalized we get
\[ g^{-\dagger}(A) \stackrel{\eqref{def_g_vonat}}{=} \left(f^{\dagger}\right)^{-\dagger}(A_0)
\stackrel{\eqref{mobius_inverzios_tetel_eq}}{=} f(A_0)=f(A)f(\zero_1)=f(A) \]

So we have constructed a series of probability measures $\probp_n$  which satisfy the consistency condition \eqref{kolmogorov_consistency_condition} thus from the Kolmogorov extension theorem it follows that there exists a random infinite array $\mathbf{X}$ with distribution $\probp$ such that
\begin{equation*}
\forall \, k \in \N \;\; \forall A \in \A_k \quad
f^{\dagger}(A)=\prob{ \mathbf{X}|_k =A}, \quad f(A)=\prob{ \mathbf{X}|_k \geq A}
\end{equation*}
holds where $\mathbf{X}|_k:=\randarray{}{}{k}$. $\rndarray{}{}$ is exchangeable because the value
 $f^{\dagger}(A)$ is invariant under relabeling. Now we show that $\rndarray{}{}$ is dissociated:
by Definition \ref{def_dissociated} we only need to check that $\mathbf{X}|_{V_1}$ is
independent of $\mathbf{X}|_{V_2}$ if $V_1, V_2 \subseteq \N$, $\abs{V_1}<+\infty$, $\abs{V_2}<+\infty$ and
  $V_1 \cap V_2=\emptyset$.
   Let $A_1 \in \A_{V_1}$, $A_2 \in \A_{V_2}$ and denote by $F_1$, $F_2$ the corresponding multigraphs. It follows from the multiplicativity of $f$ that the events $\{ \mathbf{X}|_{V_1} \geq A_1 \}$
  and $\{ \mathbf{X}|_{V_2} \geq A_2 \}$ are independent:
  \begin{multline*}
  \prob{ \{\mathbf{X}|_{V_1} \geq A_1 \} \cap \{ \mathbf{X}|_{V_2} \geq A_2 \}}
  \stackrel{\eqref{adjmatrix_of_F1F2}}{=}f(F_1 F_2)=\\ f(F_1)f(F_2)=
  \prob{ \mathbf{X}|_{V_1} \geq A_1 }  \prob{ \mathbf{X}|_{V_2} \geq A_2 }.
  \end{multline*}
Now for $i=1,2$ the family of events $\mathcal{I}_i:=\left( \{   \mathbf{X}|_{V_i} \geq A \} \right)_{A \in \A_{V_i}}$ is stable under finite intersection and $\sigma(\mathcal{I}_i)=\sigma(\mathbf{X}|_{V_i})$ by the inclusion-exclusion formula \eqref{kiejti_kiveve_ha_egyezik}, thus $\mathbf{X}|_{V_1}$ and $\mathbf{X}|_{V_2}$ are independent  by Lemma 4.2 of \cite{williams}.

We generate a (random) multigraph sequence $G_1, G_2, \dots$ by defining the adjacency matrix of $G_n$ to be
$\mathbf{X}|_n $ for all $n$. We claim that
\begin{equation}\label{t<=_f_as}
 \hominj{F}{G_n} \stackrel{\text{a.s.}}{\longrightarrow} f(F) .
\end{equation}
From  \eqref{homdensities_converge_together} and \eqref{t<=_f_as}  it follows that $\lim\limits_{n\to \infty} \homsur{F}{G_n} = f(F)$ almost surely, so $f\in \T$.

To prove \eqref{t<=_f_as} for a particular multigraph $F$ with adjacency matrix $A \in \A_k$
recall the definition of the $\sigma$-algebras $\mathcal{S}_n$ and $\mathcal{S}_{\infty}$ (see \eqref{symmetric_function_formula} and
 \eqref{def_symmetric_sigma_algebra}) and L\'evy's 'Downward' Theorem \eqref{downward}:
\begin{equation}
\condexpect{\indd{A\leq \rndarray{}{}|_k} }{\mathcal{S}_n} \stackrel{\text{a.s.}}{\longrightarrow} \condexpect{\indd{A\leq \rndarray{}{}|_k}}{\mathcal{S}_{\infty}}
\end{equation}

% a reverse martingale $\left(M_n\right)_{n=k}^{\infty}$ on the probability space
% $\left( \A_{\N} , \mathcal{F}, \probp\right)$ adapted to the decreasing
%filtration $\left( \mathcal{S}_n \right)_{n=1}^{\infty}$
%for every $n\geq k$ (see \eqref{symmetric_function_formula})
%\begin{equation}\label{def_M_reverse_martingale}
%M_n := \condexpect{\indd{A\leq \rndarray{}{}|_k} }{\mathcal{S}_n},
%\end{equation}
% where $\rndarray{}{}|_k = \randarray{}{}{k}$.
  %$M_n$ is indeed a reverse martingale: the property $\condexpect{M_n}{\mathcal{S}_{n+1}}=M_{n+1}$ directly follows from the tower property of conditional expectations.

%Recall the convergence theorem of reverse martingales (\cite{williams}, Theorem 14.4):
% if $(M_n, \mathcal{S}_n)_{n=k}^{\infty}$ is a reverse martingale  then $M_n \to M_{\infty} = \condexpect{M_k}{\mathcal{S}_{\infty}}$ almost surely %as $n \to \infty$, where $\mathcal{S}_{\infty}$ is defined by \eqref{def_symmetric_sigma_algebra}. ezt majd a masik fejezetbe
%Now we apply \eqref{downward} to show that $M_n \to M_{\infty}=\condexpect{ $ almost surely
In our case $\mathcal{S}_{\infty}$ is the trivial $\sigma$-algebra by the \ref{dissociated}. $\implies$ \ref{sym_triv}. implication of Theorem \ref{theorem_dissociated_equivalent}, thus
$\condexpect{\indd{A\leq \rndarray{}{}|_k}}{\mathcal{S}_{\infty}}=
 \prob{A\leq \rndarray{}{}|_k}=f(A)$.
 % = \sum_{A' \geq A} f^{\dagger}(A') = \left(f^{\dagger}\right)^{-\dagger}(A)= f(A)\]

So it is enough to show that for $n\geq k$ we have $\condexpect{\indd{A\leq \rndarray{}{}|_k} }{\mathcal{S}_n} = \hominj{F}{G_n}$ in order to prove \eqref{t<=_f_as}. Recalling \eqref{graphhom_roviden} and \eqref{hom_inj} we have
\begin{equation}\label{gyima}
  \hominj{F}{G_n}= \frac{1}{n(n-1)\dots (n-k+1)} \sum_{\phiinj} \graphhom{A}{\rndarray{}{}|_n}.
\end{equation}

It directly follows from the vertex exchangeability of $\rndarray{}{}$,  \eqref{def_prop_of_cond_expectation} and \eqref{symmetric_function_formula}
 that for all $\phiinj$ we have
\begin{equation*}
\condexpect{\graphhom{A}{\rndarray{}{}|_n}}{\mathcal{S}_n}=\condexpect{ \ind_{\leq} [ A, \rndarray{}{}|_n, \mathrm{id}_{[k]}  }{\mathcal{S}_n}
\end{equation*}
where $\mathrm{id}_{[k]}$ is the identity map on $[k]$.
Thus
\begin{align*}
&\condexpect{\indd{A\leq \rndarray{}{}|_k} }{\mathcal{S}_n}= \condexpect{ \ind_{\leq} [ A, \rndarray{}{}|_n, \mathrm{id}_{[k]}  }{\mathcal{S}_n} \\
%&\stackrel{(\ast)}{=} \condexpect{ \indd{\forall i,j\leq k: A(i,j)\leq \rndarray{}{}|_n(\rndperm(i),\rndperm(j))} }{\mathcal{S}_n} \\
%&= \condexpect{ \frac{(n-k)!}{n!} \sum_{\rndperm|_k} \indd{\forall i,j\leq k: A(i,j)\leq \rndarray{}{}|_n(\rndperm|_k(i),\rndperm|_k(j))} }{\mathcal{S}_n} \\
&= \frac{1}{n(n-1)\dots (n-k+1)} \sum_{\phiinj} \condexpect{  \graphhom{A}{\rndarray{}{}|_n} }{\mathcal{S}_n} \\
& \stackrel{\eqref{gyima}}{=} \condexpect{ \hominj{F}{G_n}}{\mathcal{S}_n}  \stackrel{(\ast)}{=} \hominj{F}{G_n}
\end{align*}
The equality $(\ast)$ is true since \hominj{F}{G_n} is $\mathcal{S}_n$-measurable by \eqref{symmetric_function_formula}. This concludes the proof.

\subsection{Tightness}\label{section_tightness}

In this subsection we prove Proposition \ref{proposition_tightness}. We use the notion of
tightness of a sequence of $\R^d$-valued random variables
(for the general definition see Section 5. of \cite{billingsley}): If $\mathbf{Y}_n$ is an $\R^d$-valued
random variable for each $n \in \N$ then we say that
 $\left( \mathbf{Y} \right)_{n=1}^{\infty}$ is tight if
\begin{equation*}
\lim_{m \to \infty} \max_{n}\;
 \prob{m \leq ||\mathbf{Y}_n||_{\infty}}=0.
\end{equation*}

In order to prove \eqref{convergent->tight}, assume given a convergent
sequence $\left(W_{n} \right)_{n=1}^{\infty}$ of multigraphons (see Definition \ref{def_graf_konvergencia}).
First note that the proof of
 Lemma \ref{lemma_graphkonv_konvindist} works without any change
 for multigraphons as well, so the sequence
 of random arrays $\left(\mathbf{X}_{W_n}\right)_{n=1}^{\infty}$  converges in distribution
(see Definition \ref{def_X_W},
 \eqref{eq_eobankonv} and \eqref{nemdefektiv}). In particular, the sequences
 of $\N_0$-valued random variables
 $\left(X_{W_n}(1,1)\right)_{n=1}^{\infty}$ and $\left(X_{W_n}(1,2)\right)_{n=1}^{\infty}$
converge in distribution, which implies (see Theorem 5.2 of \cite{billingsley}) that they are tight:
\begin{equation}\label{diag_nondiag_tight}
\lim_{m \to \infty} \max_{n}\;
 \prob{m \leq X_{W_n}(1,1)}=0 \qquad \lim_{m \to \infty} \max_{n} \;
 \prob{m \leq X_{W_n}(1,2)}=0
\end{equation}
Now by Definition \ref{def_X_W} we have
\begin{align}
\label{nondiag_prob_W}
\prob{m \leq X_{W_n}(1,2)}&= \int_0^1 \int_0^1 \sum_{k=m}^{\infty} W_n(x,y,k) \, \mathrm{d} x \, \mathrm{d} y \\
\label{diag_prob_W}
\prob{m \leq X_{W_n}(1,1)}&=  \int_0^1  \sum_{k=m}^{\infty} W_n(x,x,k)  \, \mathrm{d} x
\end{align}
from which \eqref{W_tight_nondiag} and \eqref{W_tight_diag} directly follow.

$ $

Now we prove \eqref{tight->convergent}.
Assume given a tight sequence of multigraphons $\left( W_{n} \right)_{n=1}^{\infty}$.
 First we prove that for all $k \in \N$ the sequence of
$\A_k$-valued
random variables $\left( X_{W_n}(i,j) \right)_{i,j=1}^k$ is tight:
\begin{equation} \label{tight_kosar}
\lim_{m \to \infty} \max_{n}\;
 \prob{m \leq \max_{i,j \in [k]} \left( X_{W_n}(i,j) \right)  }=0
\end{equation}
By  \eqref{W_tight_nondiag}, \eqref{W_tight_diag},
\eqref{nondiag_prob_W} and \eqref{diag_prob_W} we get \eqref{diag_nondiag_tight}.
Using vertex exchangeability of $\mathbf{X}_{W_n}$ we get
\[  \prob{m \leq \max_{i,j \in [k]} \left( X_{W_n}(i,j) \right)} \leq
k \cdot \prob{m \leq X_{W_n}(1,1)}
+ \binom{k}{2} \cdot \prob{m \leq X_{W_n}(1,2)}
 \]
which implies \eqref{tight_kosar}. Now by Prohorov's theorem (\cite{billingsley}, Theorem 5.1)
for each $k$ the sequence
$\left( X_{W_1}(i,j) \right)_{i,j=1}^k, \left( X_{W_2}(i,j) \right)_{i,j=1}^k, \dots$ has a subsequence
which converges in distribution. By a diagonal argument one can choose a subsequence $\left(n(m)\right)_{m=1}^{\infty}$
such that for all $k$ $\left( X_{W_{n(m)}}(i,j) \right)_{i,j=1}^k$ converges in distribution as $m \to \infty$, thus
$\mathbf{X}_{W_{n(m)}} \toindis \mathbf{X}$ for some random array $\mathbf{X}$. By the multigraphon version of
 Lemma \ref{lemma_graphkonv_konvindist} the sequence of multigraphons $\left(W_{n(m)}\right)_{m=1}^{\infty}$ converges.

%%%%%%%%%%%%%%%%%%%%%%%%%%%%%%%%%%%%%%%%%%%%%%%%%%%%%%%%%%%%%%%%%%%%%%%%
%%%References
%%%%%%%%%%%%%%%%%%%%%%%%%%%%%%%%%%%%%%%%%%%%%%%%%%%%%%%%%%%%%%%%%%%%%%%%

\end{document}